\documentclass[12pt]{amsart}
\usepackage{amsfonts,amssymb,amsmath,amscd,amsthm, amstext,bm,color}

\usepackage[colorlinks, linkcolor=blue, citecolor=blue, urlcolor=blue,
pagebackref,hypertexnames=false]{hyperref}
\usepackage{enumerate}
\usepackage{graphicx}
\usepackage{epstopdf}
\usepackage{fancyhdr}
\usepackage{geometry}
\usepackage{mathrsfs}
\usepackage{txfonts}
\usepackage{tikz}
\usetikzlibrary{arrows.meta, positioning, decorations.pathreplacing}
\usepackage{url}

\usepackage{comment}
\allowdisplaybreaks
\numberwithin{equation}{section}
\newtheorem{theorem}{Theorem}[section]
\newtheorem{lemma}[theorem]{Lemma}
\newtheorem{proposition}[theorem]{Proposition}

\newtheorem{definition}[theorem]{Definition}
\newtheorem{remark}[theorem]{Remark}

\newcommand{\Rp}{\mathbb R_+}
\newcommand{\dd}{\,d}
\newcommand{\one}{\mathbf 1}

\newcommand{\pvs}{\operatorname{p.v.}}
\newcommand{\BMO}{\operatorname{BMO}}
\newcommand{\VMO}{\operatorname{VMO}}
\newcommand{\Crhoc}{C_{\rho,c}^1(\Rp)}

\newcommand{\drho}{\dd\rho_\alpha}

\newcommand{\KAm}{K_\alpha^m}
\newcommand{\KArho}{K_\alpha^\rho}

\title[Two-weight Bessel commutators with Andersen--Kerman weights]
{Two-weight commutators for the Bessel Riesz transform with Andersen--Kerman weights}

\author[J. Li]{Ji Li}
\address{Ji Li, School of Mathematical and Physical Sciences, Macquarie University, NSW 2109, Australia}
\email{ji.li@mq.edu.au}

\author[C.-W. Liang]{Chong-Wei Liang}
\address{Chong-Wei Liang, Department of Mathematics, National Taiwan University, Taiwan}
\email{d10221001@ntu.edu.tw}

\author[C.J. Wen]{Chaojie Wen}
\address{Chaojie Wen, Department of Mathematics, Sun Yat-sen University,
Guangzhou, 510275, P.R.~China}
\email{wenchj@mail2.sysu.edu.cn}

\author[L. Wu]{Liangchuan Wu}
\address{Liangchuan Wu, School of Mathematical Sciences, Anhui University,
Hefei, 230601, P.R.~China}
\email{wuliangchuan@ahu.edu.cn}

\date{}

\begin{document}

\begin{abstract}
Let $\alpha>-1/2$, $\alpha\ne0$, and let
$$
        \Delta_\alpha=-\frac{d^2}{dx^2}-\frac{2\alpha}{x}\frac{d}{dx}
$$
be the Bessel operator on $\Rp=(0,\infty)$.  We characterize boundedness and
compactness of commutators of $R_\alpha=\frac{d}{dx}\Delta_\alpha^{-1/2}$ on the
Andersen--Kerman two-weight setting.  For $1<p<\infty$ and
$\mu,\lambda\in A_{p,\alpha}$, put
$
        \nu=\left(\frac{\mu}{\lambda}\right)^{1/p},
        \ d\rho_\alpha(x)=x^{2\alpha+1}\,dx.
$
For real-valued symbols,
$$
        \|[b,R_\alpha]\|_{L^p(\mu\,dx)\to L^p(\lambda\,dx)}
        \simeq \|b\|_{\BMO_{\nu,\alpha}},
$$
where the Bloom oscillation is computed with respect to $d\rho_\alpha$.  Moreover
$$
        [b,R_\alpha]:L^p(\mu\,dx)\to L^p(\lambda\,dx)
        \text{ is compact}
        \quad\Longleftrightarrow\quad
        b\in\VMO_{\nu,\alpha}.
$$
The proof uses the exact conjugation
$
        U_w(x)=x^{p-2\alpha-1}w(x),
        \ [U_w]_{A_p(\rho_\alpha)}=[w]_{A_{p,\alpha}},
$
which transfers the problem to the quotient Bessel kernel on
$(\Rp,|x-y|,\rho_\alpha)$.  We isolate the required Calder\'on--Zygmund estimates
and one-sided non-degeneracy in this quotient normalization.  The measure $d\rho_\alpha$
is distinct from the usual Bessel measure $dm_\alpha=x^{2\alpha}\,dx$. It is the
Bloom base measure selected by the Andersen--Kerman conjugation and is used
throughout the two-weight theory below.
\end{abstract}

\maketitle

\section{Introduction}

The Bessel operator
$$
        \Delta_\alpha=-\frac{d^2}{dx^2}-\frac{2\alpha}{x}\frac{d}{dx}
        =-x^{-2\alpha}\frac{d}{dx}x^{2\alpha}\frac{d}{dx},
        \qquad x>0,
$$
was studied by Muckenhoupt and Stein in their work on conjugate functions for
classical expansions \cite{MS65}.  The associated Riesz transform is
$$
        R_\alpha=\frac{d}{dx}\Delta_\alpha^{-1/2}.
$$
The usual Bessel measure is $dm_\alpha(x)=x^{2\alpha}\,dx$, and much of the
Bessel Calder\'on--Zygmund and commutator theory is formulated on
$(\Rp,|x-y|,dm_\alpha)$, see for instance \cite{BHNV10,CS14,DuoLiWickYang17}.
The Andersen--Kerman theorem concerns a different normalization.  Andersen and
Kerman studied the same transform on the Lebesgue weighted space
$L^p(w\,dx)$ and identified the exact class of admissible weights
\cite{AK81,k78}.

In the notation used here, a weight $w$ belongs to the Andersen--Kerman class
$A_{p,\alpha}$ if
\begin{equation}\label{eq:intro-AK}
 [w]_{A_{p,\alpha}}
 :=
 \sup_{I\subset\Rp}
 \left(\frac{1}{\rho_\alpha(I)}\int_I x^p w(x)\,dx\right)
 \left(\frac{1}{\rho_\alpha(I)}
        \int_I x^{2\alpha p'}w(x)^{-1/(p-1)}\,dx\right)^{p-1}
 <\infty,
\end{equation}
where $1<p<\infty$ and
$
        d\rho_\alpha(x)=x^{2\alpha+1}\,dx.
$
The Andersen--Kerman theorem \cite{AK81} gives
$$
        R_\alpha:L^p(w\,dx)\to L^p(w\,dx)
        \quad\text{bounded}
        \quad\Longleftrightarrow\quad
        w\in A_{p,\alpha}.
$$
Thus the natural measure behind the Andersen--Kerman weighted theory is not
$dm_\alpha$, but $d\rho_\alpha$.  The case $\alpha=0$ is excluded from the
present paper.  The conjugation below remains formally meaningful at
$\alpha=0$, but the available quotient-kernel estimates and the non-degeneracy
argument used here are for $\alpha\ne0$.

{ {Throughout the paper, a locally integrable symbol means a
function which is integrable with respect to $d\rho_\alpha$ on every bounded
interval $I\subset\Rp$, including intervals of the form $(0,r)$.  This
endpoint condition is stronger than ordinary local integrability on the open
half-line.  No corresponding finiteness of $w\,dx$ on $(0,r)$ is imposed on
an Andersen--Kerman weight unless it is stated explicitly.}}

The purpose of this paper is to prove the two-weight Bloom boundedness and
compactness theorems for $R_\alpha$ in the Andersen--Kerman normalization.  Let
$\mu,\lambda\in A_{p,\alpha}$ and set
$
        \nu=\left(\frac{\mu}{\lambda}\right)^{1/p}.
$
{ {For a locally integrable symbol $b$ and $\varepsilon>0$, let
$[b,R_{\alpha,\varepsilon}]$ be the commutator formed with the standard
kernel truncation of $R_\alpha$.  For $f\in C_c^\infty(\Rp)$, write
\[
        [b,R_\alpha]f
        =\lim_{\varepsilon\downarrow0}[b,R_{\alpha,\varepsilon}]f
        =bR_\alpha f-R_\alpha(bf)
\]
whenever the limit exists.  Boundedness or compactness of $[b,R_\alpha]$
means that this initially defined operator has a bounded or compact extension
between the indicated weighted spaces.  Such an extension is unique because
$C_c^\infty(\Rp)$ is dense in these spaces.}}
The first main result identifies the operator norm with the Bloom oscillation
measured against $d\rho_\alpha$.

\begin{theorem}\label{thm:main-bounded}
Let $\alpha>-1/2$, $\alpha\ne0$, $1<p<\infty$, and let
$\mu,\lambda\in A_{p,\alpha}$.  Put
$
        \nu=\left(\frac{\mu}{\lambda}\right)^{1/p}.
$
{ {Then, for every real-valued locally integrable $b$, the operator
$
        [b,R_\alpha]:L^p(\Rp,\mu\,dx)\longrightarrow
        L^p(\Rp,\lambda\,dx)
$
has a bounded extension if and only if $b\in\BMO_{\nu,\alpha}$.  In that case,
\begin{equation}\label{eq:main-bounded}
        \|[b,R_\alpha]\|_{L^p(\Rp,\mu\,dx)\to L^p(\Rp,\lambda\,dx)}
        \simeq
        \|b\|_{\BMO_{\nu,\alpha}}.
\end{equation}
}}
Here $\BMO_{\nu,\alpha}$ is the Bloom space with respect to
$d\rho_\alpha=x^{2\alpha+1}dx$:
$$
        \|b\|_{\BMO_{\nu,\alpha}}
        =\sup_{I\subset\Rp}\frac{1}{\int_I\nu\,d\rho_\alpha}
          \int_I |b-b_I^\rho|\,d\rho_\alpha,
        \qquad
        b_I^\rho=\frac{1}{\rho_\alpha(I)}\int_I b\,d\rho_\alpha.
$$
The constants depend only on $p$, $\alpha$,
$[\mu]_{A_{p,\alpha}}$, and $[\lambda]_{A_{p,\alpha}}$.
\end{theorem}

We point out that the difference between  $A_{p,\alpha}$ and  $A_p(dm_\alpha)$  is already visible for power weights. From \eqref{eq:intro-AK}, we obtain that
$        x^\sigma \in A_{p,\alpha} \ \Longleftrightarrow\  -1-p<\sigma<2\alpha p+p-1.
$ By contrast, converting the usual $A_p(dm_\alpha)$ theorem to the Lebesgue norm requires
$w(x)x^{-2\alpha}\in A_p(dm_\alpha)$ with $w(x)=x^\sigma$ gives
$        -1<\sigma<2\alpha+(2\alpha+1)(p-1) =2\alpha p+p-1.
$
The upper endpoint is the same, while the lower endpoint is not. Hence $A_{p,\alpha}$ contains weights with stronger singularity at the
origin than those obtained from the standard Bessel $A_p$ theorem after the change of measure. We refer to Section 7 in \cite{Wen} for more details on the comparison.
Thus, we point out that our main result 
Theorem \ref{thm:main-bounded} is not a direct consequence of \cite{DuoGongKuffnerLiWickYang21}.

The compactness theorem uses the corresponding weighted VMO space.  If
$$
        \omega_\nu(b;I)
        =\frac{1}{\int_I\nu\,d\rho_\alpha}
          \int_I |b-b_I^\rho|\,d\rho_\alpha,
$$
then $\VMO_{\nu,\alpha}$ consists of all $b\in\BMO_{\nu,\alpha}$ for which
$\omega_\nu(b;I)$ vanishes on small intervals, on large intervals, and on
intervals escaping to infinity.  The precise definition is given in
Section~\ref{sec:bmo-vmo}.

\begin{theorem}\label{thm:main-compact}
Let $\alpha>-1/2$, $\alpha\ne0$, $1<p<\infty$, and let
$\mu,\lambda\in A_{p,\alpha}$.  Put
$\nu=(\mu/\lambda)^{1/p}$.  { {Then, for every real-valued locally integrable $b$,
$        [b,R_\alpha]:L^p(\Rp,\mu\,dx)\to L^p(\Rp,\lambda\,dx)
        \ \text{extends to a compact operator}
$ if and only if $b\in\VMO_{\nu,\alpha}$.}}
\end{theorem}

The mechanism behind the theorems is an exact change of measure.  For a weight
$w$ define
$
        U_w(x)=x^{p-2\alpha-1}w(x).
$
Then
$
        [U_w]_{A_p(\rho_\alpha)}=[w]_{A_{p,\alpha}}.
$
Moreover, if $f=xF$, then
$
        \|f\|_{L^p(w\,dx)}=
        \|F\|_{L^p(U_w\,d\rho_\alpha)}.
$
The transform itself becomes
$$
        \mathcal R_\alpha F(x)=\frac1x R_\alpha(yF(y))(x),
$$
and the commutator identity is
$$
        \frac1x[b,R_\alpha]f=[b,\mathcal R_\alpha]F.
$$
Consequently the Andersen--Kerman two-weight problem is converted, without
loss of constants, into a Bloom problem for $\mathcal R_\alpha$ on the space of
homogeneous type $(\Rp,|x-y|,\rho_\alpha)$.  For the two weights in the main
theorems, the Bloom weight is unchanged by the conjugation:
$$
        \left(\frac{U_\mu}{U_\lambda}\right)^{1/p}
        =\left(\frac{\mu}{\lambda}\right)^{1/p}=\nu.
$$
What changes is the measure with respect to which oscillation is computed.

The selected base measure is a genuine part of the statement.  In the unweighted
case, the spaces $\BMO(\rho_\alpha)$ and $\BMO(m_\alpha)$ are equivalent.
After the Bloom weight $\nu=(\mu/\lambda)^{1/p}$ is inserted, the Andersen--Kerman
conjugation gives
\[
        \sup_I\frac{1}{\int_I\nu\,d\rho_\alpha}
        \int_I |b-b_I^\rho|\,d\rho_\alpha.
\]
This expression involves $\nu\,d\rho_\alpha$, or equivalently
$x\nu\,dm_\alpha$.  { {If $\nu\,dm_\alpha$ is finite on every
interval $(0,r)$, then the corresponding usual Bessel weighted BMO seminorm is
equivalent to this one, see
Proposition~\ref{prop:weighted-rho-m-equivalence}.  The full
Andersen--Kerman hypotheses do not guarantee this extra endpoint finiteness, as
shown in Proposition~\ref{prop:rho-not-m}.}}

Two kernel facts are needed after conjugation.  First, the kernel
$$
        K_\alpha^\rho(x,y)=\frac1x K_\alpha^m(x,y)
$$
of $\mathcal R_\alpha$ with respect to $d\rho_\alpha$ is a Calder\'on--Zygmund
kernel on $(\Rp,|x-y|,\rho_\alpha)$.  Second, it is non-degenerate in the
one-sided form required for the lower-bound and compactness-necessity
arguments.  These facts are not obtained by merely dividing the usual Bessel
estimates by $x$, and the $x$-smoothness in the far region involves the quotient cancellation
$$
        \partial_x\left(\frac{K_\alpha^m(x,y)}{x}\right)
        =\frac{x\partial_xK_\alpha^m(x,y)-K_\alpha^m(x,y)}{x^2}.
$$
Section~\ref{sec:kernel} isolates this quotient-kernel package and proves the
companion-interval non-degeneracy used later.

The boundedness proof then follows the modern Bloom scheme for
Calder\'on--Zygmund operators: upper bounds are obtained from sparse domination,
and lower bounds use the one-sided kernel non-degeneracy.  The compactness
proof has two additional components.  First, functions in $\VMO_{\nu,\alpha}$
are approximated by compactly supported functions which are Lipschitz in the
$\rho_\alpha$ geometry.  Second, if the VMO condition fails, the companion
intervals furnished by the quotient-kernel non-degeneracy give a weakly null
testing sequence whose commutator images have a fixed lower bound.  This is the
part where the small, large, and far-away interval regimes in the definition of
$\VMO_{\nu,\alpha}$ must be preserved by the companion construction.

These results are in the lineage of the Coifman--Rochberg--Weiss commutator
theorem, Bloom's theorem for the Hilbert transform, the Holmes--Lacey--Wick
two-weight theorem for Euclidean Riesz transforms, and related two-weight
commutator results on spaces of homogeneous type
\cite{CRW76,Bloom85,HLW,DuoGongKuffnerLiWickYang21}.  The compactness argument
is also connected with the weighted VMO method of Lacey and Li \cite{LL22}.
The present paper supplies the corresponding boundedness and compactness theory
on the Andersen--Kerman Bessel setting for $\alpha\ne0$, including the conjugated
measure, the quotient-kernel estimates, and the distinction from the usual Bessel
Bloom normalization.

The paper is organized as follows.  Section~\ref{sec:conjugation} gives the
exact Andersen--Kerman conjugation.  Section~\ref{sec:bmo-vmo} defines the
Bloom and VMO spaces, proves the density result, and compares the resulting
normalization with the usual Bessel BMO space.  Section~\ref{sec:kernel} proves
the quotient-kernel estimates and non-degeneracy.
Section~\ref{sec:Bloom-theorem} proves the abstract Bloom theorem on the
conjugated space and then derives the Andersen--Kerman boundedness theorem.
Section~\ref{sec:compactness} proves the compactness characterization.

We use one notation convention throughout.  The letter $\lambda$ denotes the
target weight, not the Bessel parameter. The Bessel parameter is always denoted
by $\alpha$.

\section{The Andersen--Kerman conjugation}\label{sec:conjugation}

{ {Throughout this section, $\alpha>-1/2$.}}
For $x>0$ and $r>0$, write
$
        B(x,r)=(x-r,x+r)\cap\Rp.
$
Unless otherwise stated, every interval in this paper is a non-empty bounded
interval of $\Rp$.
Throughout the paper
$$
        d\rho_\alpha(x)=x^{2\alpha+1}\,dx.
$$
This is not the usual Bessel measure $dm_\alpha=x^{2\alpha}\,dx$.  It is the
measure which appears after the Andersen--Kerman normalization.

\begin{lemma}\label{lem:volume}
For every $x>0$ and $r>0$,
$        \rho_\alpha(B(x,r))\simeq_\alpha r(x+r)^{2\alpha+1}.
$
Consequently $(\Rp,|x-y|,\rho_\alpha)$ is a space of homogeneous type.
\end{lemma}

{ {For an interval $I\subset\Rp$ and a function $h$ integrable
with respect to $d\rho_\alpha$ on every bounded interval, write}}
$$
        \langle h\rangle_I=\frac{1}{\rho_\alpha(I)}\int_I h\,d\rho_\alpha.
$$

\begin{definition}
Let $1<p<\infty$.  A positive measurable function $U$ belongs to
$A_p(\rho_\alpha)$ if
$$
        [U]_{A_p(\rho_\alpha)}
        =\sup_{I\subset\Rp}
        \left(\frac{1}{\rho_\alpha(I)}\int_I U\,d\rho_\alpha\right)
        \left(\frac{1}{\rho_\alpha(I)}\int_I U^{-1/(p-1)}\,d\rho_\alpha\right)^{p-1}
        <\infty.
$$
The conjugate weight is $U'=U^{1-p'}$.
\end{definition}

\begin{definition}
Let $1<p<\infty$.  A weight $w$ belongs to the Andersen--Kerman class
$A_{p,\alpha}$ if \eqref{eq:intro-AK} is finite.
\end{definition}

{ {In this definition, and throughout the Andersen--Kerman part of
the paper, a weight means an almost everywhere positive measurable function
for which the quantities in the displayed formulas are finite.  We do not
separately require $w\,dx$ to be finite on intervals $(0,r)$.  What the
Andersen--Kerman condition gives is precisely that
$$
        U_w\,d\rho_\alpha=x^{p-2\alpha-1}w\,d\rho_\alpha
$$
is finite on every bounded interval and is an $A_p(\rho_\alpha)$ measure after
conjugation.  This convention is harmless for the operator theorem and is
necessary when $\alpha\le0$.  For instance, the weight
$w_0(x)=x^{2\alpha-1}$ used in the $L^2(\rho_\alpha)$ verification below is not
integrable with respect to $dx$ on $(0,r)$ when $\alpha\le0$, but it satisfies
$U_{w_0}\equiv1$.}}

\begin{proposition}[Prop 2.6, \cite{Wen}]\label{prop:exact-Ap}
Let $1<p<\infty$ and define
$
        U_w(x)=x^{p-2\alpha-1}w(x).
$
Then
\begin{equation}\label{eq:exact-Ap}
        [U_w]_{A_p(\rho_\alpha)}=[w]_{A_{p,\alpha}}.
\end{equation}
In particular, $w\in A_{p,\alpha}$ if and only if
$U_w\in A_p(\rho_\alpha)$.
\end{proposition}

{ {For the two weights in the main theorem, set
\[
        U_\mu(x)=x^{p-2\alpha-1}\mu(x),
        \qquad
        U_\lambda(x)=x^{p-2\alpha-1}\lambda(x).
\]}}
Then
$$
        U_\mu,U_\lambda\in A_p(\rho_\alpha),
        \qquad
        \nu=\left(\frac{U_\mu}{U_\lambda}\right)^{1/p}
             =\left(\frac{\mu}{\lambda}\right)^{1/p}.
$$

Define the conjugated transform
\begin{equation}\label{eq:Rcal-def}
        \mathcal R_\alpha F(x)=\frac{1}{x}R_\alpha(yF(y))(x).
\end{equation}

\begin{proposition}[Prop 2.7, \cite{Wen}]
Let $1<p<\infty$, let $w$ be a positive measurable function,
put $U_w=x^{p-2\alpha-1}w$, and let $F$ be measurable.  Set
$f(x)=xF(x)$.  Then
\begin{equation}\label{eq:norm-input}
        \|f\|_{L^p(w\,dx)}=
        \|F\|_{L^p(U_w\,d\rho_\alpha)},
\end{equation}
with equality understood in $[0,\infty]$.  Whenever $R_\alpha f$ is defined,
\begin{equation}\label{eq:norm-output}
        \|R_\alpha f\|_{L^p(w\,dx)}=
        \|\mathcal R_\alpha F\|_{L^p(U_w\,d\rho_\alpha)}.
\end{equation}
For a locally integrable symbol $b$, define
$\mathcal R_{\alpha,\varepsilon}F=x^{-1}R_{\alpha,\varepsilon}(yF(y))$.
Then
\[
        [b,\mathcal R_{\alpha,\varepsilon}]F(x)
        =\frac1x[b,R_{\alpha,\varepsilon}]f(x).
\]
Consequently, whenever the limiting commutators exist,
\begin{equation}\label{eq:norm-comm}
        \|[b,R_\alpha]f\|_{L^p(w\,dx)}=
        \|[b,\mathcal R_\alpha]F\|_{L^p(U_w\,d\rho_\alpha)}.
\end{equation}
\end{proposition}

The preceding identities are the only passage from the Andersen--Kerman
Lebesgue formulation to the weighted Calder\'on--Zygmund formulation on
$(\Rp,|x-y|,\rho_\alpha)$.  For a pair of Andersen--Kerman weights
$\mu,\lambda\in A_{p,\alpha}$ one has
\begin{align*}
        \mu,\lambda\in A_{p,\alpha}
        &\Longleftrightarrow U_\mu,U_\lambda\in A_p(\rho_\alpha),\\
        \left(\frac{U_\mu}{U_\lambda}\right)^{1/p}
        &=\left(\frac{\mu}{\lambda}\right)^{1/p},\\
        [b,R_\alpha]:L^p(\mu\,dx)\to L^p(\lambda\,dx)
        &\Longleftrightarrow
        [b,\mathcal R_\alpha]:L^p(U_\mu\,d\rho_\alpha)\to
        L^p(U_\lambda\,d\rho_\alpha),
\end{align*}
with equality of operator norms in the last line.  The map $F\mapsto xF$ is an
isometric isomorphism at both ends of the mapping, so compactness is preserved in both directions. 
{ {More precisely, for a weight $w$, define
\[
        J_w:L^p(U_w\,d\rho_\alpha)\longrightarrow L^p(w\,dx),
        \qquad J_wF(x)=xF(x).
\]
Then $J_w$ is an isometric isomorphism and
\[
        [b,\mathcal R_\alpha]
        =J_\lambda^{-1}[b,R_\alpha]J_\mu.
\]
Thus boundedness and compactness are preserved in both directions, with equal
operator norms.}}
The rest of the paper
therefore proves the corresponding Bloom boundedness and compactness statements
for $\mathcal R_\alpha$ on the conjugated space and then transfers them back by
these identities.

\section{Bloom BMO, VMO and Bessel comparisons}\label{sec:bmo-vmo}
This section records the function-space facts used in the boundedness and
compactness arguments.  We first develop the Bloom BMO space on
$(\Rp,|x-y|,\rho_\alpha)$ and its equivalent weighted oscillation forms.  We
then introduce the corresponding VMO space and prove the density of the natural
class of functions which are smooth and compactly supported in the
$\rho_\alpha$-coordinate.  Finally, we compare the unweighted BMO and VMO
spaces for $d\rho_\alpha$ and the usual Bessel measure $dm_\alpha$, and explain
why $d\rho_\alpha$ is the natural base measure in the full weighted
Andersen--Kerman setting.

\subsection{Bloom BMO space and basic properties}

For an interval $I\subset\Rp$, write
$$
        b_I=\frac{1}{\rho_\alpha(I)}\int_I b\,d\rho_\alpha,
        \qquad
        \nu_\alpha(I)=\int_I\nu\,d\rho_\alpha.
$$

\begin{definition}
Let $\nu$ be a weight on $(\Rp,\rho_\alpha)$.  The Bloom BMO space
$\BMO_{\nu,\alpha}$ consists of all locally integrable functions $b$ such that
\begin{equation}\label{eq:BMO-nu-def}
        \|b\|_{\BMO_{\nu,\alpha}}
        :=\sup_{I\subset\Rp}
          \frac{1}{\nu_\alpha(I)}\int_I |b-b_I|\,d\rho_\alpha<\infty.
\end{equation}
\end{definition}

The next elementary lemma is used repeatedly.  It gives the Bloom weight
algebra used below.

\begin{lemma}\label{lem:Bloom-weight-algebra}
Let $1<p<\infty$ and let $U_\mu,U_\lambda\in A_p(\rho_\alpha)$.  Let
$$
        \nu=\left(\frac{U_\mu}{U_\lambda}\right)^{1/p}.
$$
Then $\nu\in A_2(\rho_\alpha)$ and
\begin{equation}\label{eq:nu-A2}
        [\nu]_{A_2(\rho_\alpha)}
        \le [U_\mu]_{A_p(\rho_\alpha)}^{1/p}
             [U_\lambda]_{A_p(\rho_\alpha)}^{1/p}.
\end{equation}
Moreover, for every interval $I\subset\Rp$,
\begin{align}
        U_\mu(I)^{1/p}U_\lambda'(I)^{1/p'}
        &\lesssim
        \int_I \nu\,d\rho_\alpha, \label{eq:Bloom-product-ineq}\qquad
        U_\lambda(I)^{1/p}U_\mu'(I)^{1/p'}
        \lesssim
        \int_I \nu^{-1}\,d\rho_\alpha.
\end{align}
Here $U_\mu'=U_\mu^{1-p'}$ and $U_\lambda'=U_\lambda^{1-p'}$.  The constants
only depend on $p$ and on the two $A_p(\rho_\alpha)$ characteristics.
\end{lemma}

\begin{proof}
By H\"older's inequality,
$
        \langle\nu\rangle_I
        \le \langle U_\mu\rangle_I^{1/p}
             \langle U_\lambda'\rangle_I^{1/p'},
        \ 
        \langle\nu^{-1}\rangle_I
        \le \langle U_\mu'\rangle_I^{1/p'}
             \langle U_\lambda\rangle_I^{1/p}.
$
Multiplying the two estimates and using the $A_p(\rho_\alpha)$ conditions for
$U_\mu$ and $U_\lambda$ gives \eqref{eq:nu-A2}.

We prove \eqref{eq:Bloom-product-ineq}.  The $A_p(\rho_\alpha)$ condition gives
$$
        \langle U_\mu\rangle_I^{1/p}
        \lesssim \langle U_\mu'\rangle_I^{-1/p'},
        \qquad
        \langle U_\lambda'\rangle_I^{1/p'}
        \lesssim \langle U_\lambda\rangle_I^{-1/p}.
$$
Therefore
$$
        \langle U_\mu\rangle_I^{1/p}
        \langle U_\lambda'\rangle_I^{1/p'}
        \lesssim
        \frac{1}{
        \langle U_\mu'\rangle_I^{1/p'}
        \langle U_\lambda\rangle_I^{1/p}}.
$$
Since
$$
        \langle\nu^{-1}\rangle_I
        \le
        \langle U_\mu'\rangle_I^{1/p'}
        \langle U_\lambda\rangle_I^{1/p},
$$
we obtain
$$
        \langle U_\mu\rangle_I^{1/p}
        \langle U_\lambda'\rangle_I^{1/p'}
        \lesssim \frac{1}{\langle\nu^{-1}\rangle_I}
        \le \langle\nu\rangle_I.
$$
Multiplying by $\rho_\alpha(I)$ proves \eqref{eq:Bloom-product-ineq}.

The second product estimate is just the first estimate applied to the pair $(U_\lambda,U_\mu)$.  The
Bloom weight for that reversed pair is $\nu^{-1}$, so the right-hand side is
$\int_I\nu^{-1}\,d\rho_\alpha$.
\end{proof}

\begin{lemma}\label{lem:Ap-fixed-portion}
Let $1<q<\infty$ and let $w\in A_q(\rho_\alpha)$.  If $E\subset I$ is
measurable, then
\begin{equation}\label{fix portion}
        \frac{\rho_\alpha(E)}{\rho_\alpha(I)}
        \le [w]_{A_q(\rho_\alpha)}^{1/q}
        \left(\frac{w(E)}{w(I)}\right)^{1/q}.
\end{equation}
Consequently, if $\rho_\alpha(E)\ge c\rho_\alpha(I)$, then
$w(I)\lesssim_{c,q,[w]_{A_q}}w(E)$.  
\end{lemma}

\begin{proof}
By H\"older's inequality,
\begin{align*}
        \rho_\alpha(E)
        &=\int_E w^{1/q}w^{-1/q}\,d\rho_\alpha                                      
        \le w(E)^{1/q}
        \left(\int_I w^{-1/(q-1)}\,d\rho_\alpha\right)^{1/q'}.
\end{align*}
The $A_q(\rho_\alpha)$ condition gives
\begin{equation*}
        \left(\int_I w^{-1/(q-1)}\,d\rho_\alpha\right)^{1/q'}
        \le [w]_{A_q(\rho_\alpha)}^{1/q}
        \frac{\rho_\alpha(I)}{w(I)^{1/q}}.
\end{equation*}
This proves the displayed estimate \eqref{fix portion}.  

The proof is complete.
\end{proof}

\begin{proposition}\label{prop:Bloom-equivalent-forms}
Let $\alpha>-1/2$ and $1<p<\infty$, let
$u,v\in A_p(\rho_\alpha)$, and put
$\eta=(u/v)^{1/p}$.  Let $b$ be integrable with respect to $d\rho_\alpha$
on every bounded interval in $\Rp$.  Then
\begin{align}
        \|b\|_{\BMO_{\eta,\alpha}}
        &\simeq
        \sup_I
        \left(\frac{1}{u(I)}\int_I |b-b_I|^p v\,d\rho_\alpha\right)^{1/p} \label{eq:Bloom-equiv-u-v}\\
        &\simeq
        \sup_I
        \left(\frac{1}{v'(I)}\int_I |b-b_I|^{p'} u'\,d\rho_\alpha\right)^{1/p'}.\label{eq:Bloom-equiv-dual}
\end{align}
Here $u'=u^{1-p'}$, $v'=v^{1-p'}$, and the suprema are over intervals in
$\Rp$.  The equivalence constants depend only on $p$ and on the
$A_p(\rho_\alpha)$ characteristics of $u$ and $v$.
\end{proposition}
\begin{proof}

We first remove the measure $d\rho_\alpha$.  Define
\[
        \Phi(x)=\rho_\alpha((0,x))
        =\frac{x^{2\alpha+2}}{2\alpha+2},
        \qquad x>0.
\]
Since
$
        \Phi'(x)=x^{2\alpha+1},
$
we have
$
        d\Phi(x)=d\rho_\alpha(x).
$
Moreover, $\Phi$ is an increasing bijection from $\Rp$ onto $\Rp$ and
therefore maps intervals bijectively onto intervals.

Set
\[
        \tilde b=b\circ\Phi^{-1},\qquad
        \tilde u=u\circ\Phi^{-1},\qquad
        \tilde v=v\circ\Phi^{-1},\qquad
        \tilde\eta=\eta\circ\Phi^{-1}
          =\left(\frac{\tilde u}{\tilde v}\right)^{1/p}.
\]
If $I\subset\Rp$ and $J=\Phi(I)$, then
$
        |J|=\rho_\alpha(I),
$
and, with the integrals on the left taken with respect to Lebesgue measure,
\[
        \tilde u(J)=u(I),\qquad
        \tilde v(J)=v(I),\qquad
        \tilde\eta(J)=\eta(I).
\]
Furthermore,
\[
        \tilde b_J
        =\frac1{|J|}\int_J\tilde b(t)\,dt
        =\frac1{\rho_\alpha(I)}
          \int_Ib(x)\,d\rho_\alpha(x)
        =b_I.
\]
Consequently,
\begin{align}
 \frac1{\tilde\eta(J)}\int_J|\tilde b-\tilde b_J|\,dt
 &=
 \frac1{\eta(I)}\int_I|b-b_I|\,d\rho_\alpha,            \label{eq:Phi-Bloom-exact}\\
 \frac1{\tilde u(J)}\int_J|\tilde b-\tilde b_J|^p\tilde v\,dt
 &=
 \frac1{u(I)}\int_I|b-b_I|^pv\,d\rho_\alpha.                 \label{eq:Phi-testing-exact}
\end{align}
Moreover,
\begin{align}
        \|b\|_{\BMO_{\eta,\alpha}}
        =\|\widetilde b\|_{\BMO_{\widetilde\eta}(\Rp,dt)}.\label{eq:Phi-BMO-equivalence}
\end{align}
The $A_p$ characteristics are also preserved exactly:
$
        [\tilde u]_{A_p(\Rp,dt)}
        =[u]_{A_p(\rho_\alpha)},
        \ 
        [\tilde v]_{A_p(\Rp,dt)}
        =[v]_{A_p(\rho_\alpha)}.
$
By Lemma~\ref{lem:Bloom-weight-algebra},
\[
        \tilde \eta=\left(\frac {\tilde u}{\tilde v}\right)^{1/p}\in A_2(\Rp,dt),
\]
with $[\tilde \eta]_{A_2}$ controlled by the $A_p$ characteristics of $\tilde u$ and $\tilde v$.

\medskip
We now pass from the half-line to the real line.  Extend $\tilde u,\tilde v, \tilde b$ evenly:
\[
        \tilde u_{\rm e}(t)=\tilde u(|t|),\qquad
        \tilde v_{\rm e}(t)=\tilde v(|t|),\qquad
        \tilde b_{\rm e}(t)=\tilde b(|t|),
        \qquad t\in\mathbb R.
\]
Then
\[
        \tilde\eta_{\rm e}(t)
        :=\left(\frac{\tilde u_{\rm e}(t)}{\tilde v_{\rm e}(t)}\right)^{1/p}
        =\tilde\eta(|t|).
\]

We first verify the Muckenhoupt conditions after reflection. It is easy to see
$$
	[\tilde u]_{A_p(\Rp, dt)}\le [\tilde u_{\rm e}]_{A_p(\mathbb R,dt)}
$$
since every positive half-line interval is also a real-line interval. For the reverse direction, we want to show that
$$
	\left(\frac1{|J|}\int_J\tilde u_{\rm e}\right)
	 \left(
	 \frac1{|J|}\int_J\tilde u_{\rm e}^{-1/(p-1)}
 	\right)^{p-1}
	\lesssim
	[\tilde u]_{A_p(\Rp, dt)}
$$
holds for every $J\subset\mathbb{R}$.

Note that intervals contained in
either half-line reduce immediately, by reflection, to half-line intervals.
Suppose that $J=(-a,b)$ crosses the origin, and put
$
        R=\max\{a,b\}.
$
Then because $J\subset (-R,R)$ and
$
        |J|=a+b\ge R,
$
\begin{align*}
 &\left(\frac1{|J|}\int_J\tilde u_{\rm e}\right)
 \left(
 \frac1{|J|}\int_J\tilde u_{\rm e}^{-1/(p-1)}
 \right)^{p-1}\le
 \left(\frac{2}{R}\int_0^R\tilde u\right)
 \left(
 \frac{2}{R}\int_0^R\tilde u^{-1/(p-1)}
 \right)^{p-1}\le 2^p[\tilde u]_{A_p(\Rp,dt)}.
\end{align*}
Thus, 
\begin{align}\label{eq:even-Ap}
        &[\tilde u]_{A_p(\Rp, dt)}\le [\tilde u_{\rm e}]_{A_p(\mathbb R,dt)}
        \le 2^p [\tilde u]_{A_p(\Rp, dt)}.
\end{align}
{ 
The same estimate holds for $\tilde v_{\rm e}$.  In particular,
\begin{equation*}
        \tilde u_{\rm e},\tilde v_{\rm e}\in A_p(\mathbb R,dt).
\end{equation*}
By Lemma~\ref{lem:Bloom-weight-algebra},
$\tilde\eta_{\rm e}\in A_2(\mathbb R,dt)$ and
\begin{equation*}
        [\tilde\eta_{\rm e}]_{A_2(\mathbb R,dt)}
        \le
        [\tilde u_{\rm e}]_{A_p(\mathbb R,dt)}^{1/p}
        [\tilde v_{\rm e}]_{A_p(\mathbb R,dt)}^{1/p}
        \lesssim_p
        [u]_{A_p(\rho_\alpha)}^{1/p}
        [v]_{A_p(\rho_\alpha)}^{1/p}.
\end{equation*}
}

We next compare the BMO seminorms on $\Rp$ and $\mathbb R$.  The
half-line seminorms are no larger than the corresponding real-line
seminorms, since every positive half-line interval is also a real-line
interval.

{ For the converse estimate, intervals contained in either open
half-line are handled directly by restriction or reflection.  It remains to
consider an interval $J=(-a,b)$ crossing the origin.}  Again put
$
        R=\max\{a,b\}
$. Let $L=(0,R)$, then
$$
	\tilde\eta_{\rm e}(J)
	=\int_{-a}^b \tilde\eta_{\rm e} dt
	\ge \frac12 \int_{-R}^R \tilde\eta_{\rm e} dt
	=\tilde\eta_{\rm e}(L)
	=\tilde\eta(L).
$$
Moreover, note that
$$
	\int_J |\tilde b_{\rm e}-(\tilde b_{\rm e})_J| dt
	\le 2 \int_J |\tilde b_{\rm e}-c|dt
$$
holds for every constant $c$. Then setting $c=(\tilde b_{\rm e})_L$ gives
$$
	\int_J |\tilde b_{\rm e}-(\tilde b_{\rm e})_J| dt
	\le 2 \int_{-R}^R |\tilde b_{\rm e}-(\tilde b_{\rm e})_L|dt
	= 4 \int_L |\tilde b_{\rm e}-(\tilde b_{\rm e})_L|dt.
$$
Thus, we have
$$
\frac{1}{\tilde\eta_{\rm e}(J)}\int_J |\tilde b_{\rm e}-(\tilde b_{\rm e})_J| dt
	\lesssim
\frac{1}{\tilde\eta(L)}\int_L |\tilde b_{\rm e}-(\tilde b_{\rm e})_L|dt
=
\frac{1}{\tilde\eta(L)}\int_L |\tilde b-(\tilde b)_L|dt
$$
and consequently
{ 
\begin{align}
        \label{eq:BMO-equivalence-even}
        \|\tilde b_{\rm e}\|_{\BMO_{\tilde\eta_{\rm e}}(\mathbb R,dt)}
        \simeq
        \|\tilde b\|_{\BMO_{\tilde\eta}(\Rp,dt)}.
\end{align}
}

To obtain \eqref{eq:Bloom-equiv-u-v}, we use the Holmes–Lacey–Wick equivalences \cite[Theorem~4.1 and Proposition~4.3]{HLW}, which are stated in dyadic form.

It remains to pass from dyadic intervals to all intervals on $\mathbb R$.
Choose a finite family of adjacent dyadic grids
$
        \mathcal D^1,\ldots,\mathcal D^M
$
with the property that, for every bounded interval $L\subset\mathbb R$,
there are $m\in\{1,\ldots,M\}$ and $Q\in\mathcal D^m$ such that
$
        L\subset Q,\ |Q|\le C_0|L|,
$
where $C_0$ is an absolute constant.
Then for each fixed dyadic grid $\mathcal D^m$, the Holmes--Lacey--Wick dyadic Bloom equivalence gives
\begin{align*}
 \sup_{Q\in\mathcal D^m}
 \frac1{\tilde\eta_{\rm e}(Q)}
 \int_Q|\tilde b_{\rm e}-(\tilde b_{\rm e})_Q|dt
 \simeq
 \sup_{Q\in\mathcal D^m}
 \left(
 \frac1{\tilde u_{\rm e}(Q)}
 \int_Q|\tilde b_{\rm e}-(\tilde b_{\rm e})_Q|^p\tilde v_{\rm e}dt
 \right)^{1/p}.
\end{align*}

We claim that
\begin{align}
	\label{eq:Bloom-comparison}
	\sup_{I\subset\mathbb{R}}
 \frac1{\tilde\eta_{\rm e}(I)}
 \int_I|\tilde b_{\rm e}-(\tilde b_{\rm e})_I|dt
 \lesssim
 \max_{1\le m\le M}\sup_{Q\in\mathcal D^m}
 \frac1{\tilde\eta_{\rm e}(Q)}
 \int_Q|\tilde b_{\rm e}-(\tilde b_{\rm e})_Q|dt
\end{align}
and
\begin{align}
	\label{eq:testing-comparison}
\sup_{I\subset\mathbb{R}}
 \left(
 \frac1{\tilde u_{\rm e}(I)}
 \int_I|\tilde b_{\rm e}-(\tilde b_{\rm e})_I|^p\tilde v_{\rm e}dt
 \right)^{1/p}
 \lesssim
 \max_{1\le m\le M}\sup_{Q\in\mathcal D^m}
 \left(
 \frac1{\tilde u_{\rm e}(Q)}
 \int_Q|\tilde b_{\rm e}-(\tilde b_{\rm e})_Q|^p\tilde v_{\rm e}dt
 \right)^{1/p}.
\end{align}

For a bounded interval $I\subset\mathbb{R}$ fixed, we find the corresponding $m\in\{1,\ldots,M\}$ and $Q\in\mathcal D^m$ such that
$I\subset Q,\ |Q|\le C_0|I|$. We obtain that $\tilde\eta_{\rm e}(Q)\lesssim \tilde\eta_{\rm e}(I)$ and $\tilde u_{\rm e}(Q)\lesssim \tilde u_{\rm e}(I)$ because of the $A_p$ doubling property. Then \eqref{eq:Bloom-comparison} follows from
$$
\int_I|\tilde b_{\rm e}-(\tilde b_{\rm e})_I|dt\le 2\int_Q|\tilde b_{\rm e}-(\tilde b_{\rm e})_Q|dt.
$$

{ 
For \eqref{eq:testing-comparison}, it suffices to prove
\begin{equation*}
        \int_I|\tilde b_{\rm e}-(\tilde b_{\rm e})_I|^p\tilde v_{\rm e}
        \lesssim_{p,[\tilde v_{\rm e}]_{A_p(\mathbb R,dt)}}
        \int_Q|\tilde b_{\rm e}-(\tilde b_{\rm e})_Q|^p\tilde v_{\rm e}.
\end{equation*}
Indeed, H\"older's inequality gives
\begin{align*}
 |(\tilde b_{\rm e})_I-(\tilde b_{\rm e})_Q|
 &\le \frac1{|I|}\int_I
       |\tilde b_{\rm e}-(\tilde b_{\rm e})_Q|\,dt 
       \le \frac1{|I|}
 \left(\int_I|\tilde b_{\rm e}-(\tilde b_{\rm e})_Q|^p
              \tilde v_{\rm e}\,dt\right)^{1/p}
 \left(\int_I \tilde v_{\rm e}^{-1/(p-1)}\,dt\right)^{1/p'}.
\end{align*}
After raising this estimate to the power $p$, we obtain
\begin{align*}
 \tilde v_{\rm e}(I)
 |(\tilde b_{\rm e})_I-(\tilde b_{\rm e})_Q|^p 
 &\le
 \frac{\tilde v_{\rm e}(I)}{|I|^p}
 \left(\int_I\tilde v_{\rm e}^{-1/(p-1)}\,dt\right)^{p-1}
 \int_I|\tilde b_{\rm e}-(\tilde b_{\rm e})_Q|^p
         \tilde v_{\rm e}\,dt \\
 &\le
 [\tilde v_{\rm e}]_{A_p(\mathbb R,dt)}
 \int_I|\tilde b_{\rm e}-(\tilde b_{\rm e})_Q|^p
         \tilde v_{\rm e}\,dt.
\end{align*}
Consequently,
\begin{align*}
 \int_I|\tilde b_{\rm e}-(\tilde b_{\rm e})_I|^p
          \tilde v_{\rm e}\,dt 
 &\le
 2^{p-1}\int_I|\tilde b_{\rm e}-(\tilde b_{\rm e})_Q|^p
                 \tilde v_{\rm e}\,dt
 +2^{p-1}\tilde v_{\rm e}(I)
 |(\tilde b_{\rm e})_I-(\tilde b_{\rm e})_Q|^p \\
 &\lesssim_{p,[\tilde v_{\rm e}]_{A_p(\mathbb R,dt)}}
 \int_Q|\tilde b_{\rm e}-(\tilde b_{\rm e})_Q|^p
        \tilde v_{\rm e}\,dt.
\end{align*}

We also compare the testing quantities on the half-line and on the real line.
The half-line supremum is no larger than the real-line supremum.  For the
reverse estimate, fix an interval $I\subset\mathbb R$.  If $I$ is contained
in $(0,\infty)$, its testing quantity is already a half-line testing
quantity.  If $I$ is contained in $(-\infty,0)$, reflection gives the same
conclusion.

It remains to consider an interval $I=(-a,b)$ with $a,b>0$.  Put
\begin{equation*}
        R=\max\{a,b\},\qquad K=(-R,R),\qquad L=(0,R).
\end{equation*}
Then $I\subset K$ and $|K|\le2|I|$.  Since
$\tilde u_{\rm e}\in A_p(\mathbb R,dt)$, its doubling property gives
\begin{equation*}
        \tilde u_{\rm e}(K)\lesssim\tilde u_{\rm e}(I).
\end{equation*}
The centre-change estimate proved above, now applied to $I\subset K$, gives
\begin{equation*}
        \int_I|\tilde b_{\rm e}-(\tilde b_{\rm e})_I|^p
        \tilde v_{\rm e}\,dt
        \lesssim
        \int_K|\tilde b_{\rm e}-(\tilde b_{\rm e})_K|^p
        \tilde v_{\rm e}\,dt.
\end{equation*}
Therefore
\begin{align*}
 &\left(
 \frac{1}{\tilde u_{\rm e}(I)}
 \int_I|\tilde b_{\rm e}-(\tilde b_{\rm e})_I|^p
        \tilde v_{\rm e}\,dt
 \right)^{1/p} \lesssim
 \left(
 \frac{1}{\tilde u_{\rm e}(K)}
 \int_K|\tilde b_{\rm e}-(\tilde b_{\rm e})_K|^p
        \tilde v_{\rm e}\,dt
 \right)^{1/p} 
 =
 \left(
 \frac{1}{\tilde u(L)}
 \int_L|\tilde b-(\tilde b)_L|^p\tilde v\,dt
 \right)^{1/p}.
\end{align*}
Here the last equality follows from evenness:
$\tilde u_{\rm e}(K)=2\tilde u(L)$,
$(\tilde b_{\rm e})_K=(\tilde b)_L$, and the integral over $K$ is twice
the corresponding integral over $L$.  We have proved
\begin{align}
\label{eq:testing-comparison-even}
\sup_{I\subset\mathbb R}
 \left(
 \frac1{\tilde u_{\rm e}(I)}
 \int_I|\tilde b_{\rm e}-(\tilde b_{\rm e})_I|^p
        \tilde v_{\rm e}\,dt
 \right)^{1/p}
 \lesssim
 \sup_{J\subset\Rp}
 \left(
 \frac1{\tilde u(J)}
 \int_J|\tilde b-(\tilde b)_J|^p\tilde v\,dt
 \right)^{1/p}.
\end{align}
}

{ 
Combining the Holmes--Lacey--Wick dyadic Bloom equivalence with
\eqref{eq:Bloom-comparison} and
\eqref{eq:testing-comparison}, we obtain
\begin{equation*}
 \|\tilde b_{\rm e}\|_{\BMO_{\tilde\eta_{\rm e}}(\mathbb R,dt)}
 \simeq
 \sup_{I\subset\mathbb R}
 \left(
 \frac1{\tilde u_{\rm e}(I)}
 \int_I|\tilde b_{\rm e}-(\tilde b_{\rm e})_I|^p
        \tilde v_{\rm e}\,dt
 \right)^{1/p}.
\end{equation*}
Using both \eqref{eq:BMO-equivalence-even} and
\eqref{eq:testing-comparison-even}, together with the trivial reverse testing
inequality obtained by restricting to positive intervals, gives
\begin{equation*}
 \|\tilde b\|_{\BMO_{\tilde\eta}(\Rp,dt)}
 \simeq
 \sup_{J\subset\Rp}
 \left(
 \frac1{\tilde u(J)}
 \int_J|\tilde b-(\tilde b)_J|^p\tilde v\,dt
 \right)^{1/p}.
\end{equation*}
}
Finally, the identities
\eqref{eq:Phi-testing-exact} and \eqref{eq:Phi-BMO-equivalence}
 give
\[
 \|b\|_{\BMO_{\eta,\alpha}}
 \simeq
 \sup_{I\subset\Rp}
 \left(
 \frac1{u(I)}
 \int_I|b-b_I|^pv\,d\rho_\alpha
 \right)^{1/p}.
\]
This proves \eqref{eq:Bloom-equiv-u-v}.

For the dual form, recall that
$
        u'=u^{1-p'},\ v'=v^{1-p'}.
$
Since $u,v\in A_p(\rho_\alpha)$, we have
$
        u',v'\in A_{p'}(\rho_\alpha).
$
Moreover,
\begin{align*}
 \left(\frac{v'}{u'}\right)^{1/p'}
 &=
 \left(
 \frac{v^{-1/(p-1)}}{u^{-1/(p-1)}}
 \right)^{1/p'}
 =
 \left(\frac uv\right)^{1/((p-1)p')}
 =
 \left(\frac uv\right)^{1/p}
 =\eta,
\end{align*}
because $(p-1)p'=p$.  Applying the already proved first equivalence
with exponent $p'$ and with the ordered pair of weights $(v',u')$
gives
\[
 \|b\|_{\BMO_{\eta,\alpha}}
 \simeq
 \sup_{I\subset\Rp}
 \left(
 \frac1{v'(I)}
 \int_I|b-b_I|^{p'}u'\,d\rho_\alpha
 \right)^{1/p'}.
\]
This proves \eqref{eq:Bloom-equiv-dual} and completes the proof.
\end{proof}

In the Andersen--Kerman application, these equivalent forms contain no extra
power of $x$.  With $u=U_\mu$, $v=U_\lambda$, and
$\nu=(u/v)^{1/p}$, Proposition~\ref{prop:Bloom-equivalent-forms} gives
\begin{align*}
        \|b\|_{\BMO_{\nu,\alpha}}
        &\simeq
        \sup_I
        \left(\frac{1}{U_\mu(I)}
        \int_I |b-b_I|^p U_\lambda\,d\rho_\alpha\right)^{1/p}
        \simeq
        \sup_I
        \left(\frac{1}{U_\lambda'(I)}
        \int_I |b-b_I|^{p'} U_\mu'\,d\rho_\alpha\right)^{1/p'}.
\end{align*}
The original definition \eqref{eq:BMO-nu-def} is more transparent for the
lower estimate, while the two displayed forms are used in the sparse upper
estimate.  Since
$$
        \nu=\left(\frac{\mu}{\lambda}\right)^{1/p}
        =\left(\frac{U_\mu}{U_\lambda}\right)^{1/p},
$$
the common Andersen--Kerman conjugation of the two weights leaves the Bloom weight itself unchanged. Only the base measure in the oscillation changes.

Let $w\in A_2(\Rp,dt)$.  Let $\VMO_w(\Rp)$ be defined by the three conditions
obtained from Definition~\ref{def:VMO-nu} with $dt$ in place of
$d\rho_\alpha$.  

\begin{lemma}\label{lem:Phi-BMO-transfer}
Let $\alpha>-1/2$, and let
$
        \Phi(x)=\rho_\alpha((0,x))=\frac{x^{2\alpha+2}}{2\alpha+2},
        \ x>0.
$
Let $\nu$ be a weight on $(\Rp,\rho_\alpha)$, and let $b$ be integrable with
respect to $d\rho_\alpha$ on every bounded interval in $\Rp$.  Define
\[
        \widetilde b=b\circ\Phi^{-1},
        \qquad
        \widetilde\nu=\nu\circ\Phi^{-1}.
\]
Then
$
        \|b\|_{\BMO_{\nu,\alpha}}
        =\|\widetilde b\|_{\BMO_{\widetilde\nu}(\Rp)}.
$
\end{lemma}

\begin{proof}
For an interval $J\subset\Rp$, write
\[
        \widetilde\nu(J)=\int_J\widetilde\nu(t)\,dt,
        \qquad
        \widetilde b_J=\frac1{|J|}\int_J\widetilde b(t)\,dt,
\]
and set
\[
        \Omega_{\widetilde\nu}(\widetilde b;J)
        =\frac1{\widetilde\nu(J)}
         \int_J|\widetilde b-\widetilde b_J|\,dt.
\]
If $I\subset\Rp$ and $J=\Phi(I)$, then the calculation leading to
\eqref{eq:Phi-Bloom-exact}, with $\eta$ replaced by $\nu$, gives
\begin{equation}\label{eq:Phi-oscillation-exact}
        \widetilde b_J=b_I,
        \qquad
        \widetilde\nu(J)=\nu_\alpha(I),
        \qquad
        \Omega_{\widetilde\nu}(\widetilde b;J)
        =\omega_\nu(b;I).
\end{equation}
Taking the supremum over all intervals proves the BMO norm identity.
\end{proof}

\begin{lemma}\label{lem:halfline-reflection-BMO}
Let $w\in A_2(\Rp,dt)$, and let $h$ be integrable on every bounded interval
in $\Rp$.  Define the even extensions
$
        w_{\rm e}(t)=w(|t|),\ h_{\rm e}(t)=h(|t|),\ t\in\mathbb R.
$
Then $w_{\rm e}\in A_2(\mathbb R)$.  Moreover, for every even function
$g_{\rm e}(t)=g(|t|)$ with $g$ integrable on every bounded interval in $\Rp$,
\begin{equation}\label{eq:reflect-BMO-equivalence}
        \|g_{\rm e}\|_{\BMO_{w_{\rm e}}(\mathbb R)}
        \simeq \|g\|_{\BMO_w(\Rp)}.
\end{equation}
The constants depend only on the $A_2$ characteristic of $w$.  
\end{lemma}

\begin{proof}
We have already proved the equivalence of $A_p$ characteristics and BMO seminorm after even extensions in the proof of Proposition \ref{prop:Bloom-equivalent-forms}. Consequently, $w_{\rm e}\in A_2(\mathbb R)$ follows directly from \eqref{eq:even-Ap} and \eqref{eq:reflect-BMO-equivalence} follows from \eqref{eq:BMO-equivalence-even} with $p=2$.
\end{proof}

\subsection{Bloom VMO space and density of smooth functions}
\begin{definition}\label{def:VMO-nu}
For $b\in\BMO_{\nu,\alpha}$ define
$$
        \omega_\nu(b;I)
        =\frac{1}{\nu_\alpha(I)}\int_I |b-b_I|\,d\rho_\alpha.
$$
We say that $b\in\VMO_{\nu,\alpha}$ if
\begin{align*}
        \lim_{r\to0}\sup_{|I|\le r}\omega_\nu(b;I)=0,\quad
        \lim_{r\to\infty}\sup_{|I|\ge r}\omega_\nu(b;I)=0,\quad
        \lim_{R\to\infty}\sup_{I\subset(R,\infty)}\omega_\nu(b;I)=0.
\end{align*}
Here $|I|$ is the ordinary Euclidean length of $I$.
\end{definition}

\begin{lemma}\label{lem:Phi-VMO-transfer}
Let $\alpha>-1/2$, and let
$
        \Phi(x)=\rho_\alpha((0,x))=\frac{x^{2\alpha+2}}{2\alpha+2},
        \ x>0.
$
Let $\nu$ be a weight on $(\Rp,\rho_\alpha)$, and let $b$ be integrable with
respect to $d\rho_\alpha$ on every bounded interval in $\Rp$.  Define
\[
        \widetilde b=b\circ\Phi^{-1},
        \qquad
        \widetilde\nu=\nu\circ\Phi^{-1}.
\]
Then $b\in\VMO_{\nu,\alpha}$ 
if and only if { {$\widetilde b\in\VMO_{\widetilde\nu}(\Rp)$}}.
\end{lemma}

\begin{proof}
For an interval $J\subset\Rp$, write
\[
        \widetilde\nu(J)=\int_J\widetilde\nu(t)\,dt,
        \qquad
        \widetilde b_J=\frac1{|J|}\int_J\widetilde b(t)\,dt,
\]
and set
\[
        \Omega_{\widetilde\nu}(\widetilde b;J)
        =\frac1{\widetilde\nu(J)}
         \int_J|\widetilde b-\widetilde b_J|\,dt.
\]
Then \eqref{eq:Phi-oscillation-exact} follows. 
It remains to compare the three vanishing terms.  Extend $\Phi$
continuously to $[0,\infty)$ by setting $\Phi(0)=0$.  Then $\Phi$ is an
increasing homeomorphism of $[0,\infty)$ onto itself.

Assume first that $b\in\VMO_{\nu,\alpha}$, and fix $\varepsilon>0$.  There are
numbers
$
        0<r<1<L,
        \
        R>1,
$
such that
\begin{equation}\label{eq:Phi-source-VMO-thresholds}
        \omega_\nu(b;I)<\varepsilon
\end{equation}
whenever
$
        |I|\le r,
        \ \text{or}\ 
        |I|\ge L,
        \ \text{or}\ 
        I\subset(R,\infty).
$

Define
\[
        m_{r,R}
        =\min_{0\le s\le R}
          \bigl(\Phi(s+r)-\Phi(s)\bigr)>0,\quad
          M_{L,R}
        =\max_{0\le s\le R}
          \bigl(\Phi(s+L)-\Phi(s)\bigr)<\infty.
\]
Let $J=\Phi(I)$ satisfy $|J|<m_{r,R}$, where $I=(a,b)$.  If $a\ge R$,
then $I\subset(R,\infty)$.  If $a<R$ and $b-a\ge r$, then
\[
        |J|=\Phi(b)-\Phi(a)
        \ge \Phi(a+r)-\Phi(a)
        \ge m_{r,R},
\]
which is a contradiction.  Thus either $I\subset(R,\infty)$ or $|I|<r$.
Equations \eqref{eq:Phi-source-VMO-thresholds} and
\eqref{eq:Phi-oscillation-exact} therefore show that
\[
 \lim_{\delta\downarrow0}
 \sup_{J\subset\Rp,|J|\le\delta}
 \Omega_{\widetilde\nu}(\widetilde b;J)=0.
\]
Indeed, for the fixed $\varepsilon$ above one may take
$\delta=m_{r,R}/2$.

If $J=\Phi(I)$ satisfies $|J|>M_{L,R}$, then either the left endpoint of
$I$ is at least $R$, or $|I|>L$. Otherwise, writing $I=(a,b)$, we would have
\[
        |J|=\Phi(b)-\Phi(a)
        \le \Phi(a+L)-\Phi(a)
        \le M_{L,R}.
\]
It follows again from \eqref{eq:Phi-source-VMO-thresholds} and
\eqref{eq:Phi-oscillation-exact} that
\[
 \lim_{A\to\infty}
 \sup_{J\subset\Rp,|J|\ge A}
 \Omega_{\widetilde\nu}(\widetilde b;J)=0.
\]
For the fixed $\varepsilon$, it is enough to take
$A=M_{L,R}+1$.

Finally, if $J\subset(\Phi(R),\infty)$, then
$\Phi^{-1}(J)\subset(R,\infty)$.  Since $\Phi(R)\to\infty$ as
$R\to\infty$, it follows that
\[
 \lim_{T\to\infty}
 \sup_{J\subset(T,\infty)}
 \Omega_{\widetilde\nu}(\widetilde b;J)=0.
\]
Hence $\widetilde b\in\VMO_{\widetilde\nu}(\Rp)$.

Conversely, suppose that $\widetilde b\in\VMO_{\widetilde\nu}(\Rp)$, and
fix $\varepsilon>0$.  Choose
$
        0<\delta<1<A,
        \
        S>1,
$
so that
\begin{equation}\label{eq:Phi-target-VMO-thresholds}
 \Omega_{\widetilde\nu}(\widetilde b;J)<\varepsilon
\end{equation}
whenever
$
        |J|\le\delta,
        \ \text{or}\ 
        |J|\ge A,
        \ \text{or}\ 
        J\subset(S,\infty).
$

Put $R=\Phi^{-1}(S)$.  By uniform continuity of $\Phi$ on $[0,R+1]$,
there is $r\in(0,1)$ such that
\[
        |\Phi(x)-\Phi(y)|<\delta
        \quad\text{whenever}\quad
        x,y\in[0,R+1],\quad |x-y|<r.
\]
If $I=(a,b)$ satisfies $|I|<r$, then either $a\ge R$, in which case
$\Phi(I)\subset(S,\infty)$, or $a<R$, in which case $b<R+1$ and then
$|\Phi(I)|<\delta$.  Thus \eqref{eq:Phi-target-VMO-thresholds} proves that
\[
 \lim_{r\downarrow0}
 \sup_{J\subset\Rp,|J|\le r}
 \omega_{\nu}(b;J)=0.
\]

Put $\kappa=2\alpha+2>1$.  For $a,\ell\ge0$, convexity of
$x\mapsto x^\kappa$ gives
\[
        \Phi(a+\ell)-\Phi(a)
        =\frac{(a+\ell)^\kappa-a^\kappa}{\kappa}
        \ge\frac{\ell^\kappa}{\kappa}
        =\Phi(\ell).
\]
Choose $L>1$ so that $\Phi(L)>A$.  Whenever $|I|\ge L$, the interval
$J=\Phi(I)$ satisfies $|J|\ge A$.  This proves that
\[
	 \lim_{L\to\infty}
 \sup_{J\subset\Rp,|J|\ge L}
 \omega_{\nu}(b;J)=0
 \]

Finally, if $I\subset(\Phi^{-1}(S),\infty)$, then
$\Phi(I)\subset(S,\infty)$.  
Since $\Phi^{-1}(S)\to\infty$ as $S\to\infty$, the choice
$R=\Phi^{-1}(S)$ gives
\[
	 \lim_{R\to\infty}
 \sup_{J\subset(R,\infty)}
 \omega_{\nu}(b;J)=0.
 \]
Hence $b\in\VMO_{\nu,\alpha}$.
\end{proof}

\begin{proposition}
Let $X$ be either $\Rp$ or $\mathbb R$. If $w\in A_2(X,dt)$, $h\in\BMO_w(X)$ and $J\subset X$ is an interval, define
$$
        \Omega_w(h;J):=\frac1{w(J)}\int_J |h-h_J|\,dt,
        \qquad
        \Omega_{2,w}(h;J)
        :=\left(\frac1{w(J)}\int_J |h-h_J|^2w^{-1}\,dt\right)^{1/2}.
$$
with $h_J:=\frac1{|J|}\int_J h(t)\,dt$. Then we can replace $\Omega_w$ by $\Omega_{2,w}$ in
each of the three VMO conditions.   
\end{proposition}

\begin{proof}
We only give the proof when $X=\mathbb{R}$. When $X=\Rp$, the results follow by first
extending $w$ and $h$ evenly using \eqref{eq:even-Ap} and \eqref{eq:BMO-equivalence-even}, applying the results on $\mathbb R$, and then restricting back to $\Rp$.

It remains to prove that for every family
$\mathcal A$ of intervals in $\mathbb R$, there is a constant $\gamma\in(0,1)$ such that
\[
 \sup_{J\in\mathcal A}\Omega_w(h;J)
 \le
 \sup_{J\in\mathcal A}\Omega_{2,w}(h;J)
 \lesssim
 \|h\|_{\BMO_w(\mathbb R)}^{\,1-\gamma}
 \left(\sup_{J\in\mathcal A}\Omega_w(h;J)\right)^\gamma.
\]

We claim that for any $J\subset\mathbb R$,
\[
 \Omega_w(h;J)
 \le
\Omega_{2,w}(h;J)
 \lesssim
 \|h\|_{\BMO_w(\mathbb R)}^{\,1-\gamma}
 \Omega_w(h;J)^\gamma.
\]
Indeed, by the open property of Muckenhoupt classes, there is
$s\in(1,2)$ such that $w\in A_s(\mathbb R)$. Let $q=s'$, we have
$w^{1-q}\in A_q(X)$. Moreover, $w\in A_q(\mathbb R)$ since $q>2>s$.

Then applying Proposition~\ref{prop:Bloom-equivalent-forms} in $w$ and $w^{1-q}$ with the associated Bloom weight $(w/w^{1-q})^{1/q}=w$. Therefore, 
\begin{equation}\label{eq:Omega-q-global}
 \sup_{J\subset X}
 \left(
        \frac1{w(J)}
        \int_J|h-h_J|^q w^{1-q}\,dt
 \right)^{1/q}
 \lesssim \|h\|_{\BMO_w(\mathbb R)}.
\end{equation}
H\"older's inequality 
yields
\begin{align*}
 \Omega_{2,w}(h;J)^2
 &=\frac1{w(J)}\int_J |h-h_J|^2w^{-1}\,dt\\
 &\le
 \left(\frac1{w(J)}\int_J |h-h_J|\,dt\right)^{\frac{q-2}{q-1}}
 \left(\frac1{w(J)}\int_J |h-h_J|^q w^{1-q}\,dt\right)^{\frac{1}{q-1}}\\
 &\le
 \Omega_w(h;J)^{\frac{q-2}{q-1}}\,\|h\|_{\BMO_w(\mathbb R)}^{\frac{q}{q-1}}.
\end{align*}
Let $\gamma=\frac{q-2}{2(q-1)}=1-\frac s2\in (0,1)$, we obtain
\[
\Omega_w(h;J)
 \le
 \Omega_{2,w}(h;J)
 \lesssim
 \|h\|_{\BMO_w(\mathbb R)}^{\,1-\gamma}
 \Omega_w(h;J)^\gamma.
 \]
The first inequality is just Cauchy--Schwarz. Then the claim is proved. 

Consequently, for every family
$\mathcal A$ of intervals in $X$,
\[
 \sup_{J\in\mathcal A}\Omega_w(h;J)
 \le
 \sup_{J\in\mathcal A}\Omega_{2,w}(h;J)
 \lesssim
 \|h\|_{\BMO_w(X)}^{\,1-\gamma}
 \left(\sup_{J\in\mathcal A}\Omega_w(h;J)\right)^\gamma.
\]
Thus, for a fixed $h\in\BMO_w(X)$, a family of interval suprema tends to zero
with $\Omega_w$ if and only if it tends to zero with $\Omega_{2,w}$.  In
particular, the three vanishing conditions may equivalently be formulated using $\Omega_{2,w}$.
\end{proof}

\begin{lemma}\label{lem:halfline-reflection-VMO}
Let $w\in A_2(\Rp,dt)$, and let $h$ be integrable on every bounded interval
in $\Rp$.  Define the even extensions
$
        w_{\rm e}(t)=w(|t|),\ h_{\rm e}(t)=h(|t|),\ t\in\mathbb R.
$
Then 
$h\in\VMO_w(\Rp)$ if and only if $h_{\rm e}\in\VMO_{w_{\rm e}}(\mathbb R)$, where the real-line VMO space is defined by three similar conditions.
\end{lemma}

\begin{proof}
It remains to check the vanishing conditions. Suppose first that
$h\in\VMO_w(\Rp)$. The calculation leading to \eqref{eq:BMO-equivalence-even}, with $\eta$ replaced by $w$, gives that for every $J\subset\mathbb{R}$, there is an interval $L\subset\Rp$ such that $|L|\le|J|\le 2|L|$ and
\[
	\frac{1}{w_{\rm e}(J)}\int_J | h_{\rm e}-( h_{\rm e})_J| dt
	\lesssim
	\frac{1}{w(L)}\int_L | h- h_L|dt.
\]
That is, $\Omega_{w_{\rm e}}(h_{\rm e};J)\lesssim \Omega_w(h;L)$. Thus, we obtain that for every $\delta>0$,
\[
 \sup_{J\subset\mathbb R,\ |J|\le\delta}
 \Omega_{w_{\rm e}}(h_{\rm e};J)
 \lesssim
 \sup_{L\subset\Rp,\ |L|\le\delta}
 \Omega_w(h;L).
\]
Similarly, for every $A>0$,
\[
 \sup_{J\subset\mathbb R,\ |J|\ge A}
 \Omega_{w_{\rm e}}(h_{\rm e};J)
 \lesssim
 \sup_{L\subset\Rp,\ |L|\ge A/2}
 \Omega_w(h;L),
\]
These inequalities give two vanishing conditions for $h_{\rm e}$: 
{ 
\[
 \lim_{\delta\downarrow0}
 \sup_{J\subset\mathbb R,|J|\le\delta}
 \Omega_{w_{\rm e}}(h_{\rm e};J)=0,
 \qquad
 \lim_{A\to\infty}
 \sup_{J\subset\mathbb R,\ |J|\ge A}
 \Omega_{w_{\rm e}}(h_{\rm e};J)=0.
\]}

If $J\subset\mathbb R\setminus[-R,R]$, then $J$ is contained in one open half-line and thus either $J$ itself or its reflection $-J$ is an
interval $L\subset (R,\infty)$. Hence the identity
$\Omega_{w_{\rm e}}(h_{\rm e};J)=\Omega_w(h;L)$ follows from the evenness and gives
\[
 \sup_{J\subset\mathbb R\setminus[-R,R]} \Omega_{w_{\rm e}}(h_{\rm e};J)
 \le
 \sup_{L\subset(R,\infty)}\Omega_w(h;L),
\]
Thus
\[
\lim_{R\to\infty}
\sup_{J\subset\mathbb R\setminus[-R,R]} \Omega_{w_{\rm e}}(h_{\rm e};J)=0.
 \]
Combining the three conditions yields that 
$h_{\rm e} \in \VMO_{w_{\rm e}} (\mathbb{R})$.

Conversely, $h_{\rm e} \in \VMO_{w_{\rm e}} (\mathbb{R})$ directly implies $h\in\VMO_w(\Rp)$, since every interval $J\subset\Rp$ is also an interval in $\mathbb R$, and
$
        \Omega_w(h;J)=\Omega_{w_{\rm e}}(h_{\rm e};J).
$
Thus
{ {$h\in\VMO_w(\Rp)$ if and only if
$h_{\rm e}\in\VMO_{w_{\rm e}}(\mathbb R)$}}.
\end{proof}

{ 
\begin{lemma}\label{lem:real-line-VMO-density}
Let $w\in A_2(\mathbb R)$.  Then
\begin{equation}\label{eq:real-line-VMO-density}
        \VMO_w(\mathbb R)
        =
        \overline{C_c^\infty(\mathbb R)}^{\,\BMO_w(\mathbb R)}.
\end{equation}
\end{lemma}

\begin{proof}
This is the one-dimensional weighted VMO density theorem
\cite[Theorem~4.1]{LL22}.  In one dimension, the balls used there are
intervals, so the definitions of $\BMO_w(\mathbb R)$ and
$\VMO_w(\mathbb R)$ agree with the definitions used here.
\end{proof}
}

\begin{lemma}\label{lem:half-line-VMO-density}
Let $w\in A_2(\Rp,dt)$. Then
\begin{equation}\label{eq:halfline-density}
        \VMO_w(\Rp)
        =\overline{C_c^1([0,\infty))|_{\Rp}}^{\,\BMO_w(\Rp)}.
\end{equation}
\end{lemma}

\begin{proof}
We write $w(E)=\int_Ew(t)\,dt$.  We use three elementary consequences of
$w\in A_2(\Rp,dt)$.

First, $w\in A_\infty$, so there are constants $C>0$ and $\theta>0$ such that
\begin{equation}\label{eq:Ainfty-half-line}
        \frac{w(E)}{w(J)}
        \le C\left(\frac{|E|}{|J|}\right)^\theta
\end{equation}
for every interval $J\subset\Rp$ and every measurable set $E\subset J$.  Choose
an integer $m$ so large that $C2^{-m\theta}<1/2$.  Applying
\eqref{eq:Ainfty-half-line} to $E=(0,R)$ and $J=(0,2^mR)$ gives
$w((0,2^mR))\ge2w((0,R))$.  Iteration gives
\begin{equation}\label{eq:w-infinite-half-line}
        w(\Rp)=\infty.
\end{equation}
Second, $w\,dt$ is doubling on intervals.  Third, for every $L>0$ there are
constants $c_L>0$ and $\gamma_L<2$ such that
\begin{equation}\label{eq:local-lower-w}
        w(J)\ge c_L |J|^{\gamma_L}
\end{equation}
whenever $J\subset(0,L)$.  To prove this, use that
$w^{-1}\in A_2$ and hence satisfies a reverse H\"older inequality on $(0,L)$.
Thus $w^{-1}\in L^s(0,L)$ for some $s>1$.  By Cauchy's inequality and H\"older's
inequality,
$$
        |J|^2
        \le w(J)\int_J w^{-1}\,dt
        \le w(J)|J|^{1-1/s}
             \left(\int_0^L w(t)^{-s}\,dt\right)^{1/s}.
$$
This is \eqref{eq:local-lower-w} with $\gamma_L=1+1/s<2$.

We prove first that $C_c^1([0,\infty))|_{\Rp}\subset\VMO_w(\Rp)$.  Let
$\varphi\in C_c^1([0,\infty))$ and choose $M>1$ so that
$\operatorname{supp}\varphi\subset[0,M]$. 
Assume $r<M$. If $J$ does not meet the support, the oscillation is zero. Otherwise, if $|J|<r$ and meets the
support, then $J\subset(0,2M)$.  Since
$$
        \int_J |\varphi-\varphi_J|\,dt
        \le 2\|\varphi'\|_\infty |J|^2,
$$
\eqref{eq:local-lower-w} gives
$$
        \lim_{r\to 0}\sup_{|J|<r}
        \Omega_w(\varphi;J)
        \lesssim \lim_{r\to 0}r^{2-\gamma_{2M}}=0.
$$

Assume $A>2M$. If $|J|>A$ and $J$ meets $[0,M]$, then $2J$ crosses the origin and
\[
w([0,A))\lesssim w(J)
\]
since $w\,dt$ is doubling. Thus $w(J)\to\infty$ uniformly in such intervals as $A\to\infty$, which follows from \eqref{eq:w-infinite-half-line}.  
On the other
hand,
$$
        \int_J |\varphi-\varphi_J|\,dt
        \le 2\|\varphi\|_{L^1(0,\infty)}.
$$
Hence 
$
        \lim_{A\to \infty}\sup_{|J|>A}
        \Omega_w(\varphi;J)
        =0.
$

Finally, intervals contained in
$(R,\infty)$ are disjoint from the support when $R>M$, leading to
$
        \lim_{R\to \infty}\sup_{J\subset(R,\infty)}
        \Omega_w(\varphi;J)
        =0.
$
This proves
$C_c^1([0,\infty))|_{\Rp}\subset\VMO_w(\Rp)$.

The space $\VMO_w(\Rp)$ is closed in the $\BMO_w(\Rp)$ seminorm.  For
any functions $f$ and $g$ and every interval $J$,
$$
        \Omega_w(f;J)
        \le \Omega_w(f-g;J)+\Omega_w(g;J)
        \le \|f-g\|_{\BMO_w(\Rp)}+\Omega_w(g;J).
$$
Therefore the closure on the right side of \eqref{eq:halfline-density} is
contained in $\VMO_w(\Rp)$.

It remains to prove the reverse inclusion.  Let $h\in\VMO_w(\Rp)$ and extend
$h$ and $w$ evenly to
$$
        H(t)=h(|t|),\qquad w_{\rm e}(t)=w(|t|),\qquad t\in\mathbb R.
$$
By Lemma~\ref{lem:halfline-reflection-BMO} and Lemma~\ref{lem:halfline-reflection-VMO}, $w_{\rm e}\in A_2(\mathbb R)$,
$H\in\VMO_{w_{\rm e}}(\mathbb R)$, and the BMO seminorms of even functions are equivalent.

{ Apply Lemma~\ref{lem:real-line-VMO-density} to $H$.  There are functions
$H_N\in C_c^\infty(\mathbb R)$ such that
\[
        \|H-H_N\|_{\BMO_{w_{\rm e}}(\mathbb R)}\to0.
\]}

The functions $H_N$ may not be even.  Replace them by their even parts
$$
        E_N(t)=\frac{H_N(t)+H_N(-t)}2.
$$
Then $E_N\in C_c^2(\mathbb R)$ is even, and the triangle inequality together
with the evenness of $H$ gives
$$
        \|H-E_N\|_{\BMO_{w_{\rm e}}(\mathbb R)}
        \le
        \frac12\|H-H_N\|_{\BMO_{w_{\rm e}}(\mathbb R)}
        +\frac12\|H-H_N(-\cdot)\|_{\BMO_{w_{\rm e}}(\mathbb R)}
        =
        \|H-H_N\|_{\BMO_{w_{\rm e}}(\mathbb R)}\to0.
$$
Let $g_N=E_N|_{\Rp}$.  Then $g_N\in C_c^1([0,\infty))|_{\Rp}$, and
\eqref{eq:reflect-BMO-equivalence} gives
$$
        \|h-g_N\|_{\BMO_w(\Rp)}\to0.
$$
This proves the reverse inclusion and completes the proof.
\end{proof}

{ 
The one-dimensional nature of the $\Phi$-coordinate after the change of
variables is the key point in the density statement below.  We include the
proof because the endpoint $0$ on the half-line must be treated carefully.

For
$$
        \Phi(x)=\rho_\alpha((0,x))=\frac{x^{2\alpha+2}}{2\alpha+2},
$$
we use the following smooth compact class:
\begin{equation}\label{eq:Crhoc-def}
        \Crhoc
        :=\bigl\{\varphi\circ\Phi:
              \varphi\in C_c^1([0,\infty))\bigr\}.
\end{equation}
Here $C_c^1([0,\infty))$ means that $\varphi$ is $C^1$ on the closed
half-line and has compact support in $[0,\infty)$.  Thus functions in
$\Crhoc$ may be non-zero at the endpoint $x=0$.  This is a convenient way to
state the half-line analogue of the classical density theorem after the
change of variables $t=\Phi(x)$.  It also removes a possible ambiguity in the
notation $C_c^1(\Rp)$, since $\Rp$ is the open half-line in the rest of the
paper.

We do not use all compactly supported $C^1$ functions of the Euclidean
variable $x$ as the dense smooth class.  That class is not intrinsic to
$d\rho_\alpha$ in the weighted statement.  To see that, put
$\sigma=(2\alpha+2)^{-1}$ and choose $\nu$ to agree with
$\Phi(x)^\sigma$ near $0$ and to be smoothly changed to a positive constant
away from $0$.  Since $0<\sigma<1$, this is an $A_2(\rho_\alpha)$ weight
in the $\Phi$-coordinate.  If
$a(x)=x\chi(x)$ with $\chi=1$ near $0$, then the oscillation over
$I_r=(0,r)$ satisfies
$$
        \int_{I_r}|a-a_{I_r}|\,d\rho_\alpha\simeq r^{2\alpha+3},
        \qquad
        \int_{I_r}\nu\,d\rho_\alpha\simeq r^{2\alpha+3}.
$$
Thus this ordinary Euclidean $C^1$ symbol is not in
$\VMO_{\nu,\alpha}$.  The density statement below is therefore formulated in
the natural $\rho_\alpha$-coordinate smooth class.

The class \eqref{eq:Crhoc-def} is also the right class for the compactness
argument, because its elements are Lipschitz in the $\rho_\alpha$-coordinate.
If $a=\varphi\circ\Phi\in\Crhoc$, then
\begin{equation}\label{eq:rho-Lipschitz}
        |a(x)-a(y)|
        \le \|\varphi'\|_\infty |\Phi(x)-\Phi(y)|
        \le \|\varphi'\|_\infty\rho_\alpha(B(x,|x-y|)),
        \qquad x,y>0.
\end{equation}
}

\begin{proposition}\label{prop:density-VMO}
Let $\nu\in A_2(\rho_\alpha)$.  Then
\begin{equation}\label{eq:density-VMO}
        \VMO_{\nu,\alpha}
        =\overline{\Crhoc}^{\,\BMO_{\nu,\alpha}}.
\end{equation}
\end{proposition}

\begin{proof}
Use the map $\Phi$ from Lemma~\ref{lem:Phi-VMO-transfer}.  Since
$d\rho_\alpha=d\Phi$, the condition $\nu\in A_2(\rho_\alpha)$ is equivalent to
$\widetilde\nu=\nu\circ\Phi^{-1}\in A_2(\Rp,dt)$, with the same $A_2$
characteristic.  Let $\tilde b=b\circ\Phi^{-1}$.  Lemma~\ref{lem:Phi-VMO-transfer} gives the exact identity
$$
        \|b\|_{\BMO_{\nu,\alpha}}
        =\|\tilde b\|_{\BMO_{\widetilde\nu}(\Rp)}
$$
and identifies the three VMO conditions.  Thus the desired density statement is
reduced to the half-line statement through the chain
$$
        b\in\VMO_{\nu,\alpha}
        \Longleftrightarrow
        \tilde b\in\VMO_{\widetilde\nu}(\Rp)
        =\overline{C_c^1([0,\infty))|_{\Rp}}^{\,\BMO_{\widetilde\nu}(\Rp)}.
$$

If $b\in\VMO_{\nu,\alpha}$, then
$\tilde b\in\VMO_{\widetilde\nu}(\Rp)$.  By Lemma~\ref{lem:half-line-VMO-density},
there are functions $\varphi_k\in C_c^1([0,\infty))$ such that
$$
        \|\tilde b-\varphi_k\|_{\BMO_{\widetilde\nu}(\Rp)}\to0
        \text{ as } k\to\infty.
$$
Set $a_k=\varphi_k\circ\Phi$.  Then $a_k\in\Crhoc$, and the norm identity gives
$$
        \|b-a_k\|_{\BMO_{\nu,\alpha}}\to0
        \text{ as } k\to\infty.
$$
Thus $\VMO_{\nu,\alpha}$ is contained in the closure on the right side of
\eqref{eq:density-VMO}.

Conversely, let $a\in\Crhoc$.  Then
$\varphi=a\circ\Phi^{-1}\in C_c^1([0,\infty))|_{\Rp}$.  Lemma~\ref{lem:half-line-VMO-density}
gives $\varphi\in\VMO_{\widetilde\nu}(\Rp)$, and
Lemma~\ref{lem:Phi-VMO-transfer} gives $a\in\VMO_{\nu,\alpha}$.  Since
$\VMO_{\nu,\alpha}$ is closed in the $\BMO_{\nu,\alpha}$ seminorm by the same
triangle inequality used above, the closure of $\Crhoc$ is contained in
$\VMO_{\nu,\alpha}$.
\end{proof}

\subsection{Comparisons of BMO spaces }\label{subsec:ordinary-bessel-bmo}

This subsection separates two comparisons which are easy to conflate.  The first
comparison is unweighted: the BMO spaces associated with $dm_\alpha$ and
$d\rho_\alpha$ coincide.  The second comparison is a Bloom comparison: after the
Andersen--Kerman conjugation has selected $d\rho_\alpha$ as the base measure,
the two-weight seminorm also contains the Bloom measure in its denominator.  The
resulting space is therefore not obtained by replacing $d\rho_\alpha$ with the
usual Bessel measure $dm_\alpha$.

\subsubsection{The equivalence between two unweighted BMO spaces}

The usual Bessel BMO space is
\begin{equation*}
        \|b\|_{\BMO(m_\alpha)}
        =\sup_{I\subset\Rp}\frac{1}{m_\alpha(I)}
          \int_I |b-b_I^m|\,dm_\alpha,
        \qquad
        b_I^m=\frac{1}{m_\alpha(I)}\int_I b\,dm_\alpha.
\end{equation*}
The unweighted BMO space used by the conjugated operator is
\begin{equation*}
        \|b\|_{\BMO(\rho_\alpha)}
        =\sup_{I\subset\Rp}\frac{1}{\rho_\alpha(I)}
          \int_I |b-b_I^\rho|\,d\rho_\alpha.
\end{equation*}

{ We first compare the two unweighted BMO spaces directly.}
\begin{proposition}
\label{prop:unweighted-BMO-equivalence}
For every $\alpha>-1/2$,
$
        \BMO(m_\alpha)=\BMO(\rho_\alpha)
$
with equivalent norms.  
\end{proposition}

{ 
\begin{proof}
For $\beta>-1$, write $d\sigma_\beta(x)=x^\beta\,dx$.  If
$\beta_1,\beta_2>-1$, then these two power measures are mutually
$A_\infty$ on intervals: there are constants $C,\delta>0$ such that
\begin{equation}\label{eq:mutual-Ainfty-power-v20}
 \frac{\sigma_{\beta_2}(E)}{\sigma_{\beta_2}(I)}
 \le C\left(
 \frac{\sigma_{\beta_1}(E)}{\sigma_{\beta_1}(I)}
 \right)^\delta,
 \qquad E\subset I,
\end{equation}
and the reverse estimate also holds.  To see this, first compare each power
measure with Lebesgue measure.  Away from zero the density is comparable to
a constant on $I$.  If $I=(a,b)$ with $a<b/2$, then $|I|\simeq b$ and
$\sigma_\beta(I)\simeq_\beta b^{\beta+1}$.  Placing a set of fixed length
at either end of $(0,b)$ gives the two required power bounds.  Combining
them proves \eqref{eq:mutual-Ainfty-power-v20}.

Let $h\in\BMO(\sigma_{\beta_1})$ and put
$A=\|h\|_{\BMO(\sigma_{\beta_1})}$.  The John--Nirenberg inequality for
the doubling measure $\sigma_{\beta_1}$ gives
\[
 \sigma_{\beta_1}
 \bigl(\{x\in I:|h-h_I^{\sigma_{\beta_1}}|>t\}\bigr)
 \le C e^{-ct/A}\sigma_{\beta_1}(I).
\]
Using \eqref{eq:mutual-Ainfty-power-v20}, the layer-cake formula, and the
fact that an average is within a factor two of the best constant, we obtain
\begin{align*}
 \frac1{\sigma_{\beta_2}(I)}
 \int_I|h-h_I^{\sigma_{\beta_2}}|\,d\sigma_{\beta_2}
 &\le
 \frac2{\sigma_{\beta_2}(I)}
 \int_I|h-h_I^{\sigma_{\beta_1}}|\,d\sigma_{\beta_2}\\
 &\le C\int_0^\infty e^{-c\delta t/A}\,dt
 \le C_{\beta_1,\beta_2}A.
\end{align*}
The case $A=0$ is immediate.  Interchanging the two exponents proves the
reverse bound.  Take $\beta_1=2\alpha$ and
$\beta_2=2\alpha+1$.
\end{proof}
}

Based on the BMO equivalence, we further have
\begin{proposition}\label{prop:unweighted-VMO-equivalence}
For every $\alpha>-1/2$,
$
        \VMO(m_\alpha)$ coincides with $\VMO(\rho_\alpha).
$
\end{proposition}

{ 
\begin{proof}
The comparison above can be applied with the BMO seminorm restricted to the
dyadic pieces of the interval under study.  The resulting bound is a constant
times a sum of the form
\[
        \sum_{k=0}^{\infty}(k+1)2^{-k(\beta+1)}\Omega_k,
        \qquad \beta>-1,
\]
where $\Omega_k$ is an oscillation for the other power measure on an interval
of length comparable to $2^{-k}|I|$ and lying in the same part of the
half-line as $I$.  This is the dyadic-annulus form of the layer-cake argument
used in the preceding proof.

For small intervals, every interval in the sum is small.  For large
intervals, fix the number of terms first: each of these finitely many
intervals becomes large, while the remaining tail is bounded by the BMO norm
times the tail of the convergent series.  If the intervals move to infinity,
all their dyadic pieces also move to infinity.  Thus the comparison preserves
all three limits in the definition of VMO.  Applying it in both directions
proves the result.
\end{proof}
}

\subsubsection{The weighted Bloom space and the usual Bessel measure}

The preceding unweighted comparison has the following weighted form.  It is
important to separate this statement from the Andersen--Kerman admissibility
{ condition.  If the same Bloom weight gives weighted measures
which are finite on every bounded interval, including $(0,r)$, for both
backgrounds, then the two weighted BMO seminorms are equivalent.  The problem
in the full Andersen--Kerman range is that this endpoint finiteness for the
usual Bessel measure need not hold.}

The Andersen--Kerman conjugation fixes the background measure to be
$
        d\rho_\alpha(x)=x^{2\alpha+1}\,dx.
$

Indeed, for a two-weight Andersen--Kerman pair $\mu,\lambda\in A_{p,\alpha}$,
the conjugated weights are
\[
        U_\mu=x^{p-2\alpha-1}\mu,
        \qquad
        U_\lambda=x^{p-2\alpha-1}\lambda,
\]
and hence the Bloom weight selected by the conjugated Calder\'on--Zygmund
problem is
\[
        \left(\frac{U_\mu}{U_\lambda}\right)^{1/p}
        =\left(\frac{\mu}{\lambda}\right)^{1/p}=\nu.
\]
The resulting Bloom seminorm is
\begin{equation}\label{eq:BMO-rho-weighted}
        \|b\|_{\BMO_\nu^{\rho_\alpha}}
        =\sup_I \frac{1}{\int_I\nu\,d\rho_\alpha}
          \int_I |b-b_I^\rho|\,d\rho_\alpha.
\end{equation}
The same base measure $d\rho_\alpha$ is used in the mean $b_I^\rho$, in the
oscillation integral, and in the Bloom measure $\nu\,d\rho_\alpha$.

Since $d\rho_\alpha=x\,dm_\alpha$, the seminorm in
\eqref{eq:BMO-rho-weighted} can be written relative to $dm_\alpha$ as
\[
        \sup_I \frac{1}{\int_I x\nu\,dm_\alpha}
        \int_I |b-b_I^\rho|\,x\,dm_\alpha.
\]
This expression is not the same formula as the one obtained by placing the same
Bloom weight on the usual Bessel background measure:
\begin{equation}\label{eq:BMO-m-weighted}
        \|b\|_{\BMO_\nu^{m_\alpha}}
        =\sup_I \frac{1}{\int_I\nu\,dm_\alpha}
          \int_I |b-b_I^m|\,dm_\alpha.
\end{equation}
{ Nevertheless, once both weighted measures are finite on every
bounded interval, including $(0,r)$, the two formulas define the same space.}

\begin{proposition}\label{prop:weighted-rho-m-equivalence}
Let $\alpha>-1/2$.  Let $\nu$ be a positive measurable function such that both
{ $\nu\,d\rho_\alpha$ and $\nu\,dm_\alpha$ are finite on every
bounded interval, including $(0,r)$.}  Then
\[
        \BMO_\nu^{\rho_\alpha}=\BMO_\nu^{m_\alpha}
\]
with equivalent seminorms.
\end{proposition}

\begin{proof}
We prove a slightly more general statement.  For $\beta>-1$, put
$d\sigma_\beta(x)=x^\beta\,dx$.  { If
$\nu\,d\sigma_\beta$ is finite on every bounded interval, including $(0,r)$,
set}
\[
        \|f\|_{\beta,\nu}^{\#}
        =\sup_I\frac{1}{\int_I\nu\,d\sigma_\beta}
          \inf_{c\in\mathbb R}\int_I |f-c|\,d\sigma_\beta.
\]
This seminorm is equivalent, up to the factor $2$, to the corresponding
seminorm defined with the average $f_I^{\sigma_\beta}$.

Fix $\beta_1,\beta_2>-1$ and assume
$\|f\|_{\beta_1,\nu}^{\#}<\infty$. Let $I\subset\Rp$ be an interval. If $I=(a,b)$ with $b\le 2a$, then $x$ is comparable to $a$ on $I$.  Hence
\[
\int_I \nu\,d\sigma_{\beta_2}\simeq a^{\beta_2-\beta_1} \int_I \nu\,x^{\beta_1}dx=a^{\beta_2-\beta_1} \int_I \nu d\sigma_{\beta_1},
\]
and then
\[
        \inf_c\int_I |f-c|\,d\sigma_{\beta_2}
        \simeq a^{\beta_2-\beta_1} \inf_c\int_I |f-c|\,d\sigma_{\beta_1}
        \lesssim \|f\|_{\beta_1,\nu}^{\#}a^{\beta_2-\beta_1}\int_I \nu\,d\sigma_{\beta_1}
        =\|f\|_{\beta_1,\nu}^{\#}\int_I \nu\,d\sigma_{\beta_2}.
\]

{ 
Suppose next that $I=(a,b)$ with $0<a<b/2$.  Choose $N\ge1$ so that
\[
        2^{-N-1}b\le a<2^{-N}b,
\]
and set
\[
 J_k=(2^{-k-1}b,2^{-k}b),\quad 0\le k<N,
 \qquad J_N=(a,2^{-N}b).
\]
These intervals form a disjoint cover of $I$.  Put
$c_k=(f)_{J_k}^{\sigma_{\beta_1}}$ for $0\le k\le N$.  On every $J_k$,
$x$ is comparable to $2^{-k}b$, and hence
\begin{equation}\label{eq:weighted-power-local-piece}
 \int_{J_k}|f-c_k|\,d\sigma_{\beta_2}
 \lesssim
 \|f\|_{\beta_1,\nu}^{\#}
 \int_{J_k}\nu\,d\sigma_{\beta_2}.
\end{equation}

For $1\le j<N$, let $K_j=J_{j-1}\cup J_j$.  The two parts of $K_j$
have comparable measure for either power measure.  Since the average gives
at most twice the best constant in the mean oscillation,
\[
 |c_j-c_{j-1}|
 \lesssim
 \|f\|_{\beta_1,\nu}^{\#}
 \frac{\int_{K_j}\nu\,d\sigma_{\beta_1}}
      {\sigma_{\beta_1}(K_j)}.
\]
Moreover,
$\sum_{k=j}^N\sigma_{\beta_2}(J_k)\lesssim\sigma_{\beta_2}(J_j)$.
It follows, by comparing the two powers on $K_j$, that
\begin{equation}\label{eq:weighted-power-regular-links}
 \sum_{j=1}^{N-1}|c_j-c_{j-1}|
       \sum_{k=j}^N\sigma_{\beta_2}(J_k)
 \lesssim
 \|f\|_{\beta_1,\nu}^{\#}
 \int_I\nu\,d\sigma_{\beta_2}.
\end{equation}

The last interval $J_N$ may be short, so we treat it separately.  With
$K_N=J_{N-1}\cup J_N$, comparison of the two power densities on $K_N$
gives
\begin{equation}\label{eq:weighted-power-last-link}
 \sigma_{\beta_2}(J_N)|c_N-c_{N-1}|
 \lesssim
 \|f\|_{\beta_1,\nu}^{\#}
 \int_{K_N}\nu\,d\sigma_{\beta_2}.
\end{equation}
Indeed, insert $(f)_{K_N}^{\sigma_{\beta_1}}$ between the two averages, and
for the $J_N$ term divide by $\sigma_{\beta_1}(J_N)$, and for the
$J_{N-1}$ term use
$\sigma_{\beta_2}(J_N)\le C\sigma_{\beta_2}(J_{N-1})$.

By telescoping the constants $c_k$ and using
\eqref{eq:weighted-power-regular-links}--
\eqref{eq:weighted-power-last-link}, we obtain
\[
 \sum_{k=0}^N\sigma_{\beta_2}(J_k)|c_k-c_0|
 \lesssim
 \|f\|_{\beta_1,\nu}^{\#}
 \int_I\nu\,d\sigma_{\beta_2}.
\]
Together with \eqref{eq:weighted-power-local-piece}, this yields
\[
 \int_I|f-c_0|\,d\sigma_{\beta_2}
 \lesssim
 \|f\|_{\beta_1,\nu}^{\#}
 \int_I\nu\,d\sigma_{\beta_2}.
\]
Intervals of the form $(0,b)$ follow by applying this estimate to
$(a,b)$ and letting $a\downarrow0$.}

Thus $\|f\|_{\beta_2,\nu}^{\#}\lesssim\|f\|_{\beta_1,\nu}^{\#}$.  Interchanging
$\beta_1$ and $\beta_2$ gives the reverse inequality.  Taking
$\beta_1=2\alpha$ and $\beta_2=2\alpha+1$ proves the proposition.
\end{proof}

Consequently, under the local-finiteness assumption neither inclusion between
$\BMO_\nu^{\rho_\alpha}$ and $\BMO_\nu^{m_\alpha}$ can fail.  The
obstruction in the Andersen--Kerman theory is instead that the hypotheses
$\mu,\lambda\in A_{p,\alpha}$ guarantee the $d\rho_\alpha$-based Bloom measure,
but not the $dm_\alpha$-based one.  The next proposition records this point.

\begin{proposition}\label{prop:rho-not-m}
Let $1<p<\infty$ and $\alpha>-1/2$.  Let $\mu,\lambda\in A_{p,\alpha}$ and put
\[
        U_\mu=x^{p-2\alpha-1}\mu,
        \qquad
        U_\lambda=x^{p-2\alpha-1}\lambda,
        \qquad
        \nu=\left(\frac{\mu}{\lambda}\right)^{1/p}.
\]
Then $U_\mu,U_\lambda\in A_p(\rho_\alpha)$ and
\[
        \nu=\left(\frac{U_\mu}{U_\lambda}\right)^{1/p}.
\]
Consequently $\nu\in A_2(\rho_\alpha)$, and the Bloom space naturally attached
to the Andersen--Kerman pair is $\BMO_\nu^{\rho_\alpha}$.

{ The same hypotheses do not imply that $\nu\,dm_\alpha$ is finite
on every interval $(0,r)$.  There are pairs $\mu,\lambda\in A_{p,\alpha}$
such that $\nu\,d\rho_\alpha$ is finite on these intervals, whereas
$\nu\,dm_\alpha$ is not.  For such a pair, the expression in
\eqref{eq:BMO-m-weighted} is not a weighted BMO seminorm with respect to a
measure which is finite at the endpoint.  Hence}
$\BMO_\nu^{m_\alpha}$ cannot replace $\BMO_\nu^{\rho_\alpha}$ in the full
Andersen--Kerman two-weight range without an additional local-finiteness
assumption.
\end{proposition}

\begin{proof}
Proposition \ref{prop:exact-Ap} gives $U_\mu,U_\lambda\in A_p(\rho_\alpha)$.  The identity
\[
        \left(\frac{\mu}{\lambda}\right)^{1/p}
        =
        \left(\frac{x^{p-2\alpha-1}\mu}{x^{p-2\alpha-1}\lambda}\right)^{1/p}
\]
is immediate.  Lemma~\ref{lem:Bloom-weight-algebra} gives
$\nu\in A_2(\rho_\alpha)$.

It remains to check that local finiteness with respect to $dm_\alpha$ is not
forced by the Andersen--Kerman assumptions.  We use power weights.  If
$w(x)=x^\beta$, then a direct calculation gives that the Andersen--Kerman condition \eqref{eq:intro-AK} is
equivalent to
\begin{equation}\label{eq:AK-power-range}
        -p-1<\beta<2\alpha p+p-1.
\end{equation}

Choose $0<\varepsilon<p/2$ and set
$
        \lambda(x)=x^{2\alpha p+p-1-\varepsilon},
        \ 
        \mu(x)=x^{-p/2-1-\varepsilon}.
$
Both exponents lie in the range \eqref{eq:AK-power-range}, hence
$\mu,\lambda\in A_{p,\alpha}$.  Their Bloom weight is
\[
        \nu(x)=\left(\frac{\mu(x)}{\lambda(x)}\right)^{1/p}
        =x^{-2\alpha-3/2}.
\]
Therefore
\[
        \int_0^1 \nu\,d\rho_\alpha
        =\int_0^1 x^{-2\alpha-3/2}x^{2\alpha+1}\,dx
        =\int_0^1 x^{-1/2}\,dx<\infty,
\]
whereas
\[
        \int_0^1 \nu\,dm_\alpha
        =\int_0^1 x^{-2\alpha-3/2}x^{2\alpha}\,dx
        =\int_0^1 x^{-3/2}\,dx=\infty.
\]
{ Thus $\nu\,d\rho_\alpha$ is finite on intervals $(0,r)$, but
$\nu\,dm_\alpha$ is not.}  This proves the assertion. The proof is complete.
\end{proof}

The same example also shows why an $A_2(m_\alpha)$ hypothesis would be an
additional assumption.  For a power weight $\nu(x)=x^\gamma$,
\[
        \nu\in A_2(\rho_\alpha)
        \quad\Longleftrightarrow\quad
        -2\alpha-2<\gamma<2\alpha+2,
\]
whereas
\[
        \nu\in A_2(m_\alpha)
        \quad\Longleftrightarrow\quad
        -2\alpha-1<\gamma<2\alpha+1.
\]
The exponent $\gamma=-2\alpha-3/2$ belongs to the first interval and not to the
second.

\subsubsection{Non-comparability with the unweighted Bessel BMO space}

The preceding subsection compares the two possible weighted Bloom
normalizations.  This is separate from the comparison with the ordinary,
unweighted Bessel BMO class.  When $\mu=\lambda$, the Bloom weight is
$\nu\equiv1$, and the main theorem reduces to the one-weight characterization
by the unweighted BMO space associated with $d\rho_\alpha$.  By
Proposition~\ref{prop:unweighted-BMO-equivalence}, this agrees with the usual
Bessel BMO space associated with $dm_\alpha=x^{2\alpha}dx$.

The next
proposition shows that, even for Bloom weights arising from Andersen--Kerman
pairs, the weighted Bloom BMO space need not contain, nor be contained in, the
ordinary Bessel BMO space.

\begin{proposition}\label{prop:weighted-unweighted-noncomparison}
Let $\alpha>-1/2$ and let $m=2\alpha+2$.  Choose
$0<\gamma<m$ and set
$
        \nu(x)=x^\gamma.
$
Then $\nu\in A_2(\rho_\alpha)$.  Moreover, $\nu$ can be realized as an
Andersen--Kerman Bloom weight: for $p=2$, if
$$
        \mu_0(x)=x^{2\alpha-1+\gamma},
        \qquad
        \lambda_0(x)=x^{2\alpha-1-\gamma},
$$
then $\mu_0,\lambda_0\in A_{2,\alpha}$ and
$$
        \nu=\left(\frac{\mu_0}{\lambda_0}\right)^{1/2}.
$$
For this weight $\nu$,
$$
        \BMO(\rho_\alpha)\not\subset \BMO_\nu^{\rho_\alpha},
        \qquad
        \BMO_\nu^{\rho_\alpha}\not\subset \BMO(\rho_\alpha).
$$
Equivalently, since $\BMO(m_\alpha)=\BMO(\rho_\alpha)$ by
Proposition~\ref{prop:unweighted-BMO-equivalence}, the same non-comparability
holds with the usual Bessel BMO space $\BMO(m_\alpha)$ in place of
$\BMO(\rho_\alpha)$.
\end{proposition}

\begin{proof}
Recall that $x^\theta\in A_2(\rho_\alpha)$ if and only if $-m<\theta<m$.  Hence
$\nu=x^\gamma\in A_2(\rho_\alpha)$.  

Also, for $p=2$, let 
$
        U_w=x^{1-2\alpha}w.
$
Thus
$
        U_{\mu_0}(x)=x^\gamma
$
and
$
        U_{\lambda_0}(x)=x^{-\gamma}
$
both belong to $A_2(\rho_\alpha)$. Therefore
$\mu_0,\lambda_0\in A_{2,\alpha}$ by Proposition \ref{prop:exact-Ap}. Moreover, 
$$
        \left(\frac{\mu_0}{\lambda_0}\right)^{1/2}
        =x^\gamma=\nu.
$$

We first give a function in the ordinary Bessel BMO space which is not in the
weighted Bloom space.  Let
$$
        b_1(x)=\log x.
$$
The change of variables
$$
        t=\Phi(x)=\rho_\alpha((0,x))=\frac{x^m}{m}
$$
sends $d\rho_\alpha$ to Lebesgue measure $dt$, and
$$
        \log x=\frac1m\log(mt).
$$
Since $\log t\in \BMO(\Rp)$, it follows that
$b_1\in \BMO(\rho_\alpha)$.

However, $b_1\notin \BMO_\nu^{\rho_\alpha}$.  For
$I_r=(0,r)$,
$$
        (b_1)_{I_r}^{\rho}
        =
        \frac{1}{\rho_\alpha(I_r)}
        \int_{I_r}\log x\,d\rho_\alpha(x)
        =
        \log r-\frac1m.
$$
Therefore
\begin{align*}
        \int_{I_r}
        \left|\log x-(b_1)_{I_r}^{\rho}\right|\,d\rho_\alpha(x)
        =
        \int_0^r
        \left|\log x-\log r+\frac1m\right|x^{m-1}\,dx                                    
        =
        r^m
        \int_0^1
        \left|\log s+\frac1m\right|s^{m-1}\,ds                                           
        = c_m r^m,
\end{align*}
where $0<c_m<\infty$.  On the other hand,
$$
        \int_{I_r}\nu\,d\rho_\alpha
        =
        \int_0^r x^\gamma x^{m-1}\,dx
        =
        \frac{r^{m+\gamma}}{m+\gamma}.
$$
Hence
$$
        \frac{1}{\int_{I_r}\nu\,d\rho_\alpha}
        \int_{I_r}
        \left|\log x-(b_1)_{I_r}^{\rho}\right|\,d\rho_\alpha(x)
        =
        c_m(m+\gamma)r^{-\gamma}\to\infty
$$
as $r\to0$.  Thus
$$
        b_1\in\BMO(\rho_\alpha)
        \quad\text{but}\quad
        b_1\notin\BMO_\nu^{\rho_\alpha}.
$$

We next give a function in the weighted Bloom space which is not in the ordinary
Bessel BMO space.  Let
$$
        b_2(x)=x^\gamma=\nu(x).
$$
For every interval $I\subset\Rp$,
\begin{align*}
        \int_I |b_2-(b_2)_I^\rho|\,d\rho_\alpha
        &\le
        \int_I b_2\,d\rho_\alpha
        +
        (b_2)_I^\rho\,\rho_\alpha(I)                                      
        =
        2\int_I x^\gamma\,d\rho_\alpha
        =
        2\int_I \nu\,d\rho_\alpha.
\end{align*}
Therefore
$$
        \|b_2\|_{\BMO_\nu^{\rho_\alpha}}\le 2.
$$
Thus $b_2\in\BMO_\nu^{\rho_\alpha}$.

On the other hand, $b_2\notin\BMO(\rho_\alpha)$.  Taking $I_R=(0,R)$, we have
$$
        (b_2)_{I_R}^{\rho}
        =
        \frac{1}{\rho_\alpha(I_R)}
        \int_{I_R}x^\gamma\,d\rho_\alpha(x)
        =
        \frac{m}{m+\gamma}R^\gamma.
$$
Hence
\begin{align*}
        \frac1{\rho_\alpha(I_R)}
        \int_{I_R}|b_2-(b_2)_{I_R}^{\rho}|\,d\rho_\alpha
        &=
        \frac{m}{R^m}
        \int_0^R
        \left|x^\gamma-\frac{m}{m+\gamma}R^\gamma\right|x^{m-1}\,dx       
        =
        m R^\gamma
        \int_0^1
        \left|s^\gamma-\frac{m}{m+\gamma}\right|s^{m-1}\,ds.
\end{align*}
The last integral is a positive finite constant depending only on
$m$ and $\gamma$.  Therefore the mean oscillation over $(0,R)$ tends to
infinity as $R\to\infty$, and so $b_2\notin\BMO(\rho_\alpha)$.

The two examples prove both failures of inclusion.  The final statement
with $\BMO(m_\alpha)$ follows from
Proposition~\ref{prop:unweighted-BMO-equivalence}.
\end{proof}

\section{The conjugated Bessel Riesz transform}\label{sec:kernel}

Throughout this section $\alpha>-1/2$ and $\alpha\ne0$, as in the main theorems.
Let $\KAm(x,y)$ denote the kernel of $R_\alpha$ with respect to the usual Bessel
measure $dm_\alpha(y)=y^{2\alpha}\,dy$:
\begin{equation*}
        R_\alpha f(x)=\pvs\int_0^\infty \KAm(x,y)f(y)\,dm_\alpha(y).
\end{equation*}
By \eqref{eq:Rcal-def}, the kernel of the conjugated operator
$\mathcal R_\alpha$ with respect to $d\rho_\alpha(y)=y^{2\alpha+1}\,dy$ is
\begin{equation*}
        \KArho(x,y)=\frac{1}{x}\KAm(x,y).
\end{equation*}
If $f=yF$, then
\begin{equation*}
        \mathcal R_\alpha F(x)
        =\frac{1}{x}\pvs\int_0^\infty \KAm(x,y)yF(y)y^{2\alpha}\,dy
        =\pvs\int_0^\infty \KArho(x,y)F(y)\,d\rho_\alpha(y).
\end{equation*}

The point of the conjugation is not only that the underlying measure changes
from $dm_\alpha$ to $d\rho_\alpha$.  It also changes the kernel which has to be
estimated.  In the right far region $y\ge2x$, the $x$-smoothness of
$\KAm(x,y)/x$ depends on the cancellation in
$x\partial_x\KAm(x,y)-\KAm(x,y)$.  Thus the far derivative bound below is a
quotient-kernel estimate, not a consequence of separate bounds for $\KAm$ and
$\partial_x\KAm$.

The following proposition records precisely the kernel used later.

{ 

\begin{proposition}[\cite{Wen}]
\label{prop:quotient-kernel-package}
The kernel $\KArho$ is $C^1$ away from the diagonal.  There is a constant
$C_\alpha$ such that the following estimates hold.

If $x/2<y<2x$ and $x\ne y$, then
\begin{align}
        |\KArho(x,y)|
        &\le C_\alpha\frac{1}{x^{2\alpha+1}|x-y|},
        \label{eq:AK-local-conj-size}\\
        |\partial_x\KArho(x,y)|+|\partial_y\KArho(x,y)|
        &\le C_\alpha\frac{1}{x^{2\alpha+1}|x-y|^2}.
        \label{eq:AK-local-conj-deriv}
\end{align}
If $y\ge2x$, then
\begin{align}
        |\KArho(x,y)|
        &\le C_\alpha\frac{1}{y^{2\alpha+2}},
        \label{eq:AK-right-conj-size}\\
        |\partial_x\KArho(x,y)|+|\partial_y\KArho(x,y)|
        &\le C_\alpha\frac{1}{y^{2\alpha+3}}.
        \label{eq:AK-right-conj-deriv}
\end{align}
If $x\ge2y$, then
\begin{align}
        |\KArho(x,y)|
        &\le C_\alpha\frac{1}{x^{2\alpha+2}},
        \label{eq:AK-left-conj-size}\\
        |\partial_x\KArho(x,y)|+|\partial_y\KArho(x,y)|
        &\le C_\alpha\frac{1}{x^{2\alpha+3}}.
        \label{eq:AK-left-conj-deriv}
\end{align}
Finally, there are constants $\kappa_0\in(0,1)$, $\kappa_\infty>1$, and
$c_\alpha>0$ such that
\begin{align}
        |\KArho(x,y)|
        &\ge c_\alpha\frac{1}{x^{2\alpha+1}(y-x)},
        &&0<\frac{y}{x}-1<\kappa_0,
        \label{eq:local-lower-Kcal}\\
        |\KArho(x,y)|
        &\ge c_\alpha\frac{1}{y^{2\alpha+2}},
        &&y\ge \kappa_\infty x.
        \label{eq:offdiag-lower-Kcal}
\end{align}
In each of the two lower-bound regions the kernel has a fixed sign.
\end{proposition}

\begin{proof}
The kernel $\KArho$ is the kernel denoted by $\mathcal K_\alpha$ in
\cite{Wen}.  Hence
\eqref{eq:AK-local-conj-size}--\eqref{eq:AK-left-conj-deriv} follow from
Lemma~3.1 of \cite{Wen}.  It remains to prove the two lower bounds.

We first recall the homogeneity used in the proof of that lemma:
\begin{equation*}
        \KArho(rx,ry)=r^{-2\alpha-2}\KArho(x,y),
        \qquad r,x,y>0,\quad x\ne y.
\end{equation*}
We also use the local expansion from the same lemma.  On the fixed interval
$J=[1/2,2]$, there are a constant $b_\alpha\ne0$ and a remainder
$E_\alpha$ such that
\begin{equation*}
        \KArho(u,v)
        =\frac{b_\alpha}{u^{2\alpha+1}(u-v)}+E_\alpha(u,v),
        \qquad u,v\in J,\quad u\ne v,
\end{equation*}
and
\begin{equation*}
        |E_\alpha(u,v)|
        \le C_\alpha\left(1+\log^+\frac{\ell(J)}{|u-v|}\right).
\end{equation*}
Put $v=1+\delta$, where $0<\delta<1/2$.  After changing the constant in
the remainder estimate, we obtain
\begin{equation*}
        \KArho(1,1+\delta)
        =-\frac{b_\alpha}{\delta}+E_\alpha(1,1+\delta),
        \qquad
        |E_\alpha(1,1+\delta)|
        \le C_\alpha\left(1+\log\frac{1}{\delta}\right).
\end{equation*}
Since
\begin{equation*}
        \delta\left(1+\log\frac{1}{\delta}\right)\longrightarrow0
        \qquad\text{as }\delta\longrightarrow0^+,
\end{equation*}
we may choose $\kappa_0\in(0,1/2)$ so that
\begin{equation*}
        |E_\alpha(1,1+\delta)|
        \le \frac{|b_\alpha|}{2\delta},
        \qquad 0<\delta<\kappa_0.
\end{equation*}
It follows that
\begin{equation*}
        |\KArho(1,1+\delta)|
        \ge \frac{|b_\alpha|}{2\delta},
        \qquad 0<\delta<\kappa_0,
\end{equation*}
and the kernel has the sign of $-b_\alpha$ in this region.  Now let
$0<y/x-1<\kappa_0$ and put $\delta=y/x-1$.  By homogeneity,
\begin{align*}
        |\KArho(x,y)|
        &=x^{-2\alpha-2}|\KArho(1,y/x)|
        \ge \frac{|b_\alpha|}{2}
        \frac{x^{-2\alpha-2}}{y/x-1}
        =\frac{|b_\alpha|}{2}
        \frac{1}{x^{2\alpha+1}(y-x)}.
\end{align*}
This proves \eqref{eq:local-lower-Kcal}, including the fixed-sign statement.

We next prove the far lower bound.  For $0<t<1$, set
$        A_\alpha(t):=\KArho(t,1).
$
We show that $A_\alpha$ has a non-zero limit at the origin.  Put
$\nu=\alpha-1/2>-1$.  The heat kernel of $\Delta_\alpha$, with respect to
$dm_\alpha(y)=y^{2\alpha}\,dy$, is
\begin{equation*}
        W_s^\alpha(x,y)
        =\frac{(xy)^{-\nu}}{2s}
        I_\nu\left(\frac{xy}{2s}\right)
        \exp\left(-\frac{x^2+y^2}{4s}\right),
\end{equation*}
where $I_\nu$ is the modified Bessel function.  Since
$R_\alpha=\partial_x\Delta_\alpha^{-1/2}$, its kernel with respect to
$dm_\alpha$ satisfies, away from the diagonal,
\begin{equation*}
        \KAm(x,y)
        =\frac{1}{\sqrt{\pi}}\int_0^\infty
        \partial_xW_s^\alpha(x,y)\,\frac{ds}{\sqrt{s}}.
\end{equation*}
Recall also that $\KArho(x,y)=x^{-1}\KAm(x,y)$.

For each fixed $s>0$, the power series for $I_\nu$ gives
\begin{equation*}
        W_s^\alpha(t,1)
        =C_{\alpha,s}\left[
        1+\left(\frac{1}{16(\nu+1)s^2}-\frac{1}{4s}\right)t^2
        +O_{\alpha,s}(t^4)\right],
\end{equation*}
where
\begin{equation*}
        C_{\alpha,s}
        =\frac{e^{-1/(4s)}}{2\cdot4^\nu\Gamma(\nu+1)s^{\nu+1}}.
\end{equation*}
Therefore
\begin{equation*}
        \lim_{t\longrightarrow0^+}
        \frac{\partial_tW_s^\alpha(t,1)}{t}
        =2C_{\alpha,s}
        \left(\frac{1}{16(\nu+1)s^2}-\frac{1}{4s}\right).
\end{equation*}

We may pass this limit through the $s$-integral.  Indeed, the power series
for $z^{-\nu}I_\nu(z)$ near zero and its standard exponential bound for
large $z$ give, uniformly for $0<t\le1/2$,
\begin{equation*}
        \left|\frac{\partial_tW_s^\alpha(t,1)}{t}\right|
        \le C_\alpha
        \begin{cases}
        s^{-N_\alpha}e^{-c/s},&0<s\le1,\\
        s^{-\nu-2}+s^{-\nu-3},&s\ge1,
        \end{cases}
\end{equation*}
for some $N_\alpha,c>0$.  

To obtain the exponential factor in the first
line, combine the factor $e^{-(1+t^2)/(4s)}$ in the heat kernel with the
factor $e^{t/(2s)}$ from $I_\nu(t/(2s))$.
Since $t\le1/2$, their product is at most $e^{-1/(16s)}$.  After
multiplication by $s^{-1/2}$, the right-hand side is integrable because
$\nu>-1$.  

Dominated convergence now yields
\begin{align*}
        \lim_{t\longrightarrow0^+}A_\alpha(t)
        &=\lim_{t\longrightarrow0^+}\frac{\KAm(t,1)}{t}
        =\frac{1}{\sqrt{\pi}\,4^\nu\Gamma(\nu+1)}
        \left[
        \frac{4^{\nu+5/2}\Gamma(\nu+5/2)}{16(\nu+1)}
        -\frac{4^{\nu+3/2}\Gamma(\nu+3/2)}{4}
        \right]\\
        &=\frac{\Gamma(\nu+3/2)}
        {\sqrt{\pi}\,(\nu+1)\Gamma(\nu+1)}
        =\frac{\Gamma(\alpha+1)}
        {\sqrt{\pi}\,\Gamma(\alpha+3/2)}
        =:a_\alpha>0.
\end{align*}
Here we used
\begin{equation*}
        \int_0^\infty e^{-1/(4s)}s^{-q}\,ds
        =4^{q-1}\Gamma(q-1),
        \qquad q>1.
\end{equation*}

Choose $t_0\in(0,1/2]$ so that
$A_\alpha(t)\ge a_\alpha/2$ for $0<t\le t_0$, and put
$\kappa_\infty=t_0^{-1}$.  If $y\ge\kappa_\infty x$, then $x/y\le t_0$.
Homogeneity with variable $y$ gives
\begin{equation*}
        \KArho(x,y)=y^{-2\alpha-2}A_\alpha(x/y),
\end{equation*}
and hence
\begin{equation*}
        |\KArho(x,y)|
        \ge \frac{a_\alpha}{2}\frac{1}{y^{2\alpha+2}}.
\end{equation*}
The kernel is positive in this region.  This proves
\eqref{eq:offdiag-lower-Kcal}.  Taking the smaller of the two lower-bound
constants completes the proof.
\end{proof}

}

{

\begin{proposition}[\cite{Wen}, Proposition 3.2]
\label{prop:CZ-kernel}
The kernel $K_\alpha^\rho$ is a standard Calder\'on--Zygmund kernel on
$(\Rp,|x-y|,\rho_\alpha)$.  More precisely, there are constants $C_\alpha>0$ and
$\delta\in(0,1]$ such that, for $x\ne y$,
\begin{equation}\label{eq:CZ-size}
        |\KArho(x,y)|
        \le \frac{C_\alpha}{\rho_\alpha(B(x,|x-y|))},
\end{equation}
and whenever $|x-x'|\le |x-y|/2$,
\begin{equation*}
        |\KArho(x,y)-\KArho(x',y)|
        \le
        C_\alpha\left(\frac{|x-x'|}{|x-y|}\right)^\delta
        \frac{1}{\rho_\alpha(B(x,|x-y|))}.
\end{equation*}
The analogous estimate in the $y$ variable also holds.  
The operator
$\mathcal R_\alpha$ is bounded on $L^2(\rho_\alpha)$.
\end{proposition}

Based on the above argument, we note that the maximal truncation of $\mathcal R_\alpha$
is also of weak type $(1,1)$ with respect to $\rho_\alpha$ via the standard results (see for example \cite{CW77}).
}

We also need a lower bound for the kernel.  This non-degeneracy replaces the
Fourier argument used for the Euclidean Riesz transforms.

\begin{proposition}\label{prop:nondegenerate}
There are constants $3\le A_1<A_2<\infty$ and $c_\alpha>0$ with the following
property.  For every interval $I=I(x_0,r):=B(x_0,r)$ in $\Rp$, there is another
interval $\widetilde I=I(y_0,r)$ such that $\sup I<\inf\widetilde I$ and
$$
        A_1r\le |x_0-y_0|\le A_2r,
        \qquad
        \rho_\alpha(\widetilde I)\simeq_\alpha \rho_\alpha(I),
$$
and, for all $(x,y)\in I\times\widetilde I$, the kernel $\KArho(x,y)$ has one
sign and
\begin{equation}\label{eq:kernel-lower}
        |\KArho(x,y)|\ge \frac{c_\alpha}{\rho_\alpha(I)}.
\end{equation}
\end{proposition}

\begin{proof}
We use the two lower bounds in Proposition~\ref{prop:quotient-kernel-package}.  Choose
$C_0>2$ so large that
$$
        \frac{5}{C_0-1}<\kappa_0.
$$
Then choose $A>C_0+3$ so large that
$$
        A-1\ge \kappa_\infty(C_0+1).
$$

Let $I=B(x_0,r)$.  First suppose $x_0\le C_0r$.  Put $y_0=Ar$ and
$\widetilde I=B(y_0,r)$.  Since $A>C_0+3$, one has $\sup I<\inf\widetilde I$.
If $x\in I$ and $y\in\widetilde I$, then
$x\le(C_0+1)r$ and $y\ge(A-1)r$.  Hence $y\ge\kappa_\infty x$, so
\eqref{eq:offdiag-lower-Kcal} applies.  Also $y\simeq r$, and therefore
$$
        |\KArho(x,y)|\gtrsim r^{-2\alpha-2}.
$$
Since $x_0\le C_0r$, Lemma~\ref{lem:volume} gives
$\rho_\alpha(I)\simeq r^{2\alpha+2}$.  Thus
$|\KArho(x,y)|\gtrsim \rho_\alpha(I)^{-1}$.  The same volume estimate gives
$\rho_\alpha(\widetilde I)\simeq \rho_\alpha(I)$, and $|x_0-y_0|$ lies between
$(A-C_0)r$ and $Ar$.

Now suppose $x_0>C_0r$.  Put $y_0=x_0+3r$ and
$\widetilde I=B(y_0,r)$.  Then $\sup I<\inf\widetilde I$.  If $x\in I$ and
$y\in\widetilde I$, then
$$
        r\le y-x\le5r,
        \qquad
        0<\frac{y}{x}-1\le\frac{5r}{x_0-r}<\kappa_0.
$$
Thus \eqref{eq:local-lower-Kcal} applies.  Since $x\simeq x_0$ on $I$,
$$
        |\KArho(x,y)|\gtrsim \frac{1}{x_0^{2\alpha+1}r}.
$$
Because $x_0>C_0r$, Lemma~\ref{lem:volume} gives
$\rho_\alpha(I)\simeq r x_0^{2\alpha+1}$.  Hence
$|\KArho(x,y)|\gtrsim \rho_\alpha(I)^{-1}$.  Also
$\rho_\alpha(\widetilde I)\simeq\rho_\alpha(I)$, and $|x_0-y_0|=3r$.

Taking
\begin{equation*}
        A_1=3,
        \qquad
        A_2=A+C_0,
\end{equation*}
gives the stated distance bounds.  
\end{proof}

\begin{remark}\label{rem:kernel-nondeg-data}
The later lower-bound and compactness arguments use only the following
properties of the construction.  The companion interval $\widetilde I$ has
the same radius as $I$, lies to the right of $I$ at a distance comparable to
this radius, satisfies
$
        \rho_\alpha(\widetilde I)\simeq_\alpha\rho_\alpha(I),
$
and the kernel has one sign and satisfies \eqref{eq:kernel-lower} on
$I\times\widetilde I$.

The construction also preserves the three interval regimes used in
$\VMO_{\nu,\alpha}$.  Indeed,
$
        |I|\simeq r\simeq|\widetilde I|,
$
so the small-interval and large-interval regimes are preserved.  Suppose that
$I_j\subset(R_j,\infty)$ with $R_j\to\infty$.  In the second case of the
construction, $\widetilde I_j$ lies farther to the right than $I_j$.  In the
first case,
$
        x_j-r_j>R_j,
        \ 
        x_j\le C_0r_j,
$
and hence
\[
        r_j>\frac{R_j}{C_0-1},
        \qquad
        \inf\widetilde I_j=(A-1)r_j\longrightarrow\infty.
\]
Thus the far-away regime is preserved as well.
\end{remark}

\section{The Bloom theorem and the boundedness characterization}\label{sec:Bloom-theorem}

In this section $T$ denotes a Calder\'on--Zygmund operator on
$(\Rp,|x-y|,\rho_\alpha)$ with kernel $K$ satisfying the estimates in
Proposition~\ref{prop:CZ-kernel}.  The weights $u,v$ are in
$A_p(\rho_\alpha)$ and
$$
        \eta=\left(\frac{u}{v}\right)^{1/p}.
$$
The weighted BMO space is $\BMO_{\eta,\alpha}$, defined as in
\eqref{eq:BMO-nu-def} with $\eta$ in place of $\nu$.

A family $\mathcal S$ of intervals is sparse if there are measurable sets
$E_I\subset I$, $I\in\mathcal S$, such that the sets $E_I$ are pairwise
disjoint and $\rho_\alpha(E_I)\ge c\rho_\alpha(I)$ for a fixed positive
constant $c$.

We now prove the upper bound.  The argument is the standard Bloom proof using
sparse domination, written with the measure $d\rho_\alpha$.

\begin{theorem}\label{thm:Bloom-upper-general}
Let $1<p<\infty$, $u,v\in A_p(\rho_\alpha)$, and
$\eta=(u/v)^{1/p}$.  Assume that $T$ is an $L^2(\rho_\alpha)$-bounded
Calder\'on--Zygmund operator on $(\Rp,|x-y|,\rho_\alpha)$ and that its maximal
truncation has the standard weak type $(1,1)$ estimate.  If
$b\in\BMO_{\eta,\alpha}$, then
\begin{equation}\label{eq:Bloom-upper-general}
        \|[b,T]\|_{L^p(u\,d\rho_\alpha)\to L^p(v\,d\rho_\alpha)}
        \lesssim
        \|b\|_{\BMO_{\eta,\alpha}}.
\end{equation}
The implicit constant depends on $p$, $\alpha$, the Calder\'on--Zygmund
constants of $T$, the weak type constant of the maximal truncation, and the
$A_p(\rho_\alpha)$ characteristics of $u$ and $v$.
\end{theorem}

\begin{proof}
For a nonnegative function $h$ and a sparse family $\mathcal S$, set
\[
\begin{aligned}
        \mathcal T_{\mathcal S,b}h(x)
        &=
        \sum_{I\in\mathcal S}
        |b(x)-b_I|\langle h\rangle_I\mathbf 1_I(x),\qquad
        \mathcal T_{\mathcal S,b}^{*}h(x)
        =
        \sum_{I\in\mathcal S}
        \langle |b-b_I|h\rangle_I\mathbf 1_I(x).
\end{aligned}
\]
For bounded compactly supported $f$, the sparse domination theorem
\cite[Theorem~3.7]{DuoGongKuffnerLiWickYang21}, applied with $m=1$,
gives finitely many sparse families
$\mathcal S_1,\ldots,\mathcal S_N$ such that
\[
        |[b,T]f|
        \lesssim
        \sum_{j=1}^{N}
        \left(
        \mathcal T_{\mathcal S_j,b}(|f|)
        +
        \mathcal T_{\mathcal S_j,b}^{*}(|f|)
        \right)
\]
almost everywhere.

The estimates proved in
\cite[Section~4]{DuoGongKuffnerLiWickYang21}, applied with
$m=1$ and $k=0,1$, give
\[
\begin{aligned}
        &\|\mathcal T_{\mathcal S,b}h\|_{L^p(v\,d\rho_\alpha)}
        +
        \|\mathcal T_{\mathcal S,b}^{*}h\|_{L^p(v\,d\rho_\alpha)}
        \lesssim
        \|b\|_{\BMO_{\eta,\alpha}}
        \|h\|_{L^p(u\,d\rho_\alpha)}.
\end{aligned}
\]
Since the number of sparse families is fixed, it follows that
\[
        \|[b,T]f\|_{L^p(v\,d\rho_\alpha)}
        \lesssim
        \|b\|_{\BMO_{\eta,\alpha}}
        \|f\|_{L^p(u\,d\rho_\alpha)}.
\]
The general case follows by density.
\end{proof}

The lower bound uses only the non-degeneracy of the kernel.  We give the
median argument in the present normalization.

\begin{theorem}\label{thm:Bloom-lower-general}
Assume, in addition, that $T$ satisfies the non-degeneracy conclusion of
Proposition~\ref{prop:nondegenerate}.  Let $1<p<\infty$,
$u,v\in A_p(\rho_\alpha)$, and $\eta=(u/v)^{1/p}$.  If $b$ is real-valued and
$[b,T]:L^p(u\,d\rho_\alpha)\to L^p(v\,d\rho_\alpha)$ is bounded, then
$b\in\BMO_{\eta,\alpha}$ and
\begin{equation}\label{eq:Bloom-lower-general}
        \|b\|_{\BMO_{\eta,\alpha}}
        \lesssim
        \|[b,T]\|_{L^p(u\,d\rho_\alpha)\to L^p(v\,d\rho_\alpha)}.
\end{equation}
\end{theorem}

\begin{proof}

Fix an interval $I$.  Let
$\widetilde I$ be the companion interval given by
Proposition~\ref{prop:nondegenerate}.  Let $m$ be a median of $b$ on
$\widetilde I$ with
respect to $\rho_\alpha$.  Thus both sets
$$
        \{y\in\widetilde I:b(y)\le m\},
        \qquad
        \{y\in\widetilde I:b(y)\ge m\}
$$
have $\rho_\alpha$-measure at least $\rho_\alpha(\widetilde I)/2$.
Set
$$
        E_+=\{x\in I:b(x)\ge m\},
        \qquad
        E_-=\{x\in I:b(x)<m\},
$$
and choose measurable sets
$$
        F_-\subset\{y\in\widetilde I:b(y)\le m\},
        \qquad
        F_+\subset\{y\in\widetilde I:b(y)\ge m\}
$$
with
$$
        \rho_\alpha(F_-),\rho_\alpha(F_+)
        \ge \frac{1}{2}\rho_\alpha(\widetilde I).
$$
For $(x,y)\in E_+\times F_-$ we have
$b(x)-b(y)\ge0$ and
$|b(x)-b(y)|\ge |b(x)-m|$.  For
$(x,y)\in E_-\times F_+$ the same statement holds with the opposite sign.
The kernel has a fixed sign on $I\times\widetilde I$.  Therefore,
using \eqref{eq:kernel-lower},
\begin{align*}
        \int_I |b-m|\,d\rho_\alpha
        &\lesssim
        \int_{E_+}|[b,T]\one_{F_-}(x)|\,d\rho_\alpha(x)
        +\int_{E_-}|[b,T]\one_{F_+}(x)|\,d\rho_\alpha(x).
\end{align*}
For example
\begin{align*}
        |[b,T]\one_{F_-}(x)|
        &=\left|\int_{F_-}(b(x)-b(y))K(x,y)\,d\rho_\alpha(y)\right|             \\
        &\ge \frac{c}{\rho_\alpha(I)}
             \int_{F_-}|b(x)-b(y)|\,d\rho_\alpha(y)                             \\
        &\gtrsim |b(x)-m|,
        \qquad x\in E_+,
\end{align*}
because $\rho_\alpha(F_-)\simeq\rho_\alpha(I)$.

Now apply H\"older's inequality and the boundedness of $[b,T]$:
\begin{align*}
        \int_{E_+}|[b,T]\one_{F_-}|\,d\rho_\alpha
        &\le
        \|[b,T]\one_{F_-}\|_{L^p(v\,d\rho_\alpha)}
        v'(E_+)^{1/p'}  \\                                                          
        &\le
        \|[b,T]\|_{L^p(u\,d\rho_\alpha)\to L^p(v\,d\rho_\alpha)}\,u(F_-)^{1/p}v'(I)^{1/p'}\\                                               
        &\le
        \|[b,T]\|_{L^p(u\,d\rho_\alpha)\to L^p(v\,d\rho_\alpha)}\,u(\widetilde I)^{1/p}v'(I)^{1/p'}.
\end{align*}
Let $H$ be an interval containing $I\cup\widetilde I$ with
$\rho_\alpha(H)\lesssim\rho_\alpha(I)\simeq\rho_\alpha(\widetilde I)$.  Lemma~\ref{lem:Ap-fixed-portion}, applied to $I\subset H$, gives
$u(H)\lesssim u(I)$, since $\rho_\alpha(I)\simeq\rho_\alpha(H)$.  Hence
$u(\widetilde I)\le u(H)\lesssim u(I)$.
Thus
$$
        \int_{E_+}|[b,T]\one_{F_-}|\,d\rho_\alpha
        \lesssim \|[b,T]\|_{L^p(u\,d\rho_\alpha)\to L^p(v\,d\rho_\alpha)}\,u(I)^{1/p}v'(I)^{1/p'}.
$$
The same estimate holds for the $E_-$ term.  By Lemma~\ref{lem:Bloom-weight-algebra},
$$
        u(I)^{1/p}v'(I)^{1/p'}\lesssim \eta_\alpha(I),
        \qquad
        \eta_\alpha(I)=\int_I\eta\,d\rho_\alpha.
$$
Consequently
$$
        \int_I |b-m|\,d\rho_\alpha
        \lesssim \|[b,T]\|_{L^p(u\,d\rho_\alpha)\to L^p(v\,d\rho_\alpha)}\,\eta_\alpha(I).
$$
Since
$$
        \int_I |b-b_I|\,d\rho_\alpha
        \le 2\int_I |b-m|\,d\rho_\alpha,
$$
we obtain
$$
        \frac{1}{\eta_\alpha(I)}
        \int_I |b-b_I|\,d\rho_\alpha
        \lesssim \|[b,T]\|_{L^p(u\,d\rho_\alpha)\to L^p(v\,d\rho_\alpha)}.
$$
Taking the supremum over all intervals $I$ proves \eqref{eq:Bloom-lower-general}.
\end{proof}

\subsection{Proof of the Andersen--Kerman boundedness theorem}\label{subsec:proof-bounded}

We now prove Theorem~\ref{thm:main-bounded}.

\begin{proof}[Proof of Theorem~\ref{thm:main-bounded}]
Let
$
        U_\mu=x^{p-2\alpha-1}\mu,
        \
        U_\lambda=x^{p-2\alpha-1}\lambda.
$
By Proposition~\ref{prop:exact-Ap},
$U_\mu,U_\lambda\in A_p(\rho_\alpha)$ and
$
        [U_\mu]_{A_p(\rho_\alpha)}=[\mu]_{A_{p,\alpha}},
        \ 
        [U_\lambda]_{A_p(\rho_\alpha)}=[\lambda]_{A_{p,\alpha}}.
$
Moreover
$$
        \left(\frac{U_\mu}{U_\lambda}\right)^{1/p}
        =\left(\frac{\mu}{\lambda}\right)^{1/p}=\nu.
$$

If $b\in\BMO_{\nu,\alpha}$, Theorem~\ref{thm:Bloom-upper-general}, applied to
$T=\mathcal R_\alpha$, gives
$$
        \|[b,\mathcal R_\alpha]\|_{L^p(U_\mu\,d\rho_\alpha)
        \to L^p(U_\lambda\,d\rho_\alpha)}
        \lesssim \|b\|_{\BMO_{\nu,\alpha}}.
$$
The norm identity \eqref{eq:norm-comm} then gives the upper bound in
\eqref{eq:main-bounded}.

Conversely, suppose
$[b,R_\alpha]:L^p(\mu\,dx)\to L^p(\lambda\,dx)$ is bounded.  By
\eqref{eq:norm-comm},
$[b,\mathcal R_\alpha]$ is bounded from $L^p(U_\mu\,d\rho_\alpha)$ to
$L^p(U_\lambda\,d\rho_\alpha)$ with the same norm.  The kernel of
$\mathcal R_\alpha$ is non-degenerate by Proposition~\ref{prop:nondegenerate}.
The lower bound theorem, Theorem~\ref{thm:Bloom-lower-general}, gives
$$
        \|b\|_{\BMO_{\nu,\alpha}}
        \lesssim
       \|[b,\mathcal R_\alpha]\|_{L^p(U_\mu\,d\rho_\alpha)
        \to L^p(U_\lambda\,d\rho_\alpha)}.
$$
Using \eqref{eq:norm-comm} again gives the lower bound in
\eqref{eq:main-bounded}.
\end{proof}

\section{Compactness}\label{sec:compactness}

We first prove compactness on the conjugated space and then return to the
Lebesgue Andersen--Kerman setting.

\begin{lemma}\label{lem:tail-integrability}
Let $1<s<\infty$, let $w\in A_s(\rho_\alpha)$, and let $E\subset\Rp$ be a
bounded interval.  Then
\begin{equation*}
\lim_{R\to\infty}
        \int_R^\infty
        \sup_{y\in E}\frac{w(x)}{\rho_\alpha(B(x,|x-y|))^s}
        \drho(x)=0.
\end{equation*}
\end{lemma}

\begin{proof}
Use the open property of $A_s(\rho_\alpha)$ weights and choose
$q<s$ such that $w\in A_q(\rho_\alpha)$.  Let $E\subset(0,A)$.  Fix an interval
$E_0=(A+1,A+2)$ to the right of $E$.  Lemma~\ref{lem:Ap-fixed-portion}, applied to $E_0\subset B$, gives the power-growth estimate
$$
        w(B)\le C
        \left(\frac{\rho_\alpha(B)}{\rho_\alpha(E_0)}\right)^q w(E_0)
$$
whenever $B$ is an interval containing $E_0$.  
If $x\in(2^k,2^{k+1})$, $k$ is large, and $y\in E$, then
$B(x,|x-y|)$ contains $E_0$.
Lemma \ref{lem:volume} gives
$
        \rho_\alpha(B(x,|x-y|))\simeq 2^{k(2\alpha+2)}
$
 and
$
        \rho_\alpha((0,2^{k+1}))\simeq 2^{k(2\alpha+2)}.
$
Therefore
\begin{align*}
        \int_{2^k}^{2^{k+1}}
        \sup_{y\in E}\frac{w(x)}{\rho_\alpha(B(x,|x-y|))^s}
        \drho(x)
        &\lesssim 2^{-ks(2\alpha+2)}w((0,2^{k+1}))  
        \lesssim 2^{-ks(2\alpha+2)}2^{kq(2\alpha+2)}
        =2^{-k(s-q)(2\alpha+2)}.
\end{align*}
The last series is summable because $q<s$.  Its tail tends to zero.
\end{proof}

\subsection{Smooth symbols give compact commutators}

\begin{lemma}\label{lem:smooth-compact}
Let $T$ be a Calder\'on--Zygmund operator on $(\Rp,|x-y|,\rho_\alpha)$ with
kernel satisfying the estimates in Proposition~\ref{prop:CZ-kernel}.  Let
$u,v\in A_p(\rho_\alpha)$ and let $a\in\Crhoc$.  Then
$$
        [a,T]:L^p(u\,d\rho_\alpha)\to L^p(v\,d\rho_\alpha)
$$
is compact.
\end{lemma}

\begin{proof}
Write $u'=u^{1-p'}$ and $a=\varphi\circ\Phi$ with
$\varphi\in C_c^1([0,\infty))$.  Choose
$s_1>1$ so that the closed-half-line support of $a$ is contained in $[0,s_1]$.
Set
$
        S=(0,s_1),\ S^*=(0,2s_1+2),\ L_a=\|\varphi'\|_\infty.
$
Recall from \eqref{eq:rho-Lipschitz}, the basic cancellation is
\begin{equation}\label{eq:a-rho-Lip-compact}
        |a(x)-a(y)|
        \le L_a\rho_\alpha(B(x,|x-y|)),
        \qquad x,y>0.
\end{equation}
Together with the size estimate for $K$, this makes
$(a(x)-a(y))K(x,y)$ locally integrable across the diagonal.  Thus the
commutator has the integral representation
$$
        [a,T]f(x)=\int_0^\infty(a(x)-a(y))K(x,y)f(y)\,d\rho_\alpha(y)
$$
for bounded compactly supported $f$, and the estimates below extend the formula
by density to $L^p(u\,d\rho_\alpha)$.  Since $a(x)-a(y)=0$ unless at least one
of $x,y$ belongs to $S$, only the diagonal and the two tails at infinity have to
be removed.

Fix $0<\varepsilon<1$ and $R>2s_1+2$.  Put
$$
        H_{\varepsilon,R}(x,y)
        =(a(x)-a(y))K(x,y)\one_{\{|x-y|>\varepsilon\}}
        \one_{(0,R)}(x)\one_{(0,R)}(y),
$$
and let $A_{\varepsilon,R}$ be the corresponding integral operator.  On the
support of $H_{\varepsilon,R}$ the denominator
$\rho_\alpha(B(x,|x-y|))$ is bounded below by a positive constant depending on
$\alpha,\varepsilon,R$.  Hence
$$
        H_{\varepsilon,R}
        \in L^p(v\,d\rho_\alpha;L^{p'}(u'\,d\rho_\alpha)).
$$
For any kernel $H$ in this Bochner space, the associated integral operator
$A_H$ satisfies
$$
        |A_H f(x)|\le
        \|H(x,\cdot)\|_{L^{p'}(u'\,d\rho_\alpha)}
        \|f\|_{L^p(u\,d\rho_\alpha)},
$$
and hence
$$
        \|A_H\|_{L^p(u\,d\rho_\alpha)\to L^p(v\,d\rho_\alpha)}
        \le \|H\|_{L^p(v\,d\rho_\alpha;L^{p'}(u'\,d\rho_\alpha))}.
$$
Simple tensors are dense in the Bochner space.  If
$G(x,y)=\sum_{m=1}^N \phi_m(x)\psi_m(y)$ is such a tensor kernel, then
$$
        A_Gf(x)=\sum_{m=1}^N \phi_m(x)\int_0^\infty \psi_m(y)f(y)\,d\rho_\alpha(y)
$$
is finite rank.  Taking tensor kernels converging to $H_{\varepsilon,R}$ in the
Bochner norm and using the preceding operator-norm estimate shows that
$A_{\varepsilon,R}$ is compact.

We now estimate $[a,T]-A_{\varepsilon,R}$ by three errors.

First, the near-diagonal part is
$$
        E_{\varepsilon}^0f(x)
        =\int_{\{|x-y|\le\varepsilon\}}(a(x)-a(y))K(x,y)f(y)\,d\rho_\alpha(y).
$$
If the integrand is non-zero, then one of $x,y$ lies in $S$. Since
$\varepsilon<1$, both variables lie in $S^*$.  By
\eqref{eq:a-rho-Lip-compact} and \eqref{eq:CZ-size},
$$
        |a(x)-a(y)|\,|K(x,y)|
        \lesssim L_a\one_{S^*}(x)\one_{B(x,\varepsilon)\cap S^*}(y).
$$
H\"older's inequality and Fubini give
\begin{align*}
        |E_{\varepsilon}^0f(x)|^p
        &\lesssim L_a^p\one_{S^*}(x)
        u'(B(x,\varepsilon)\cap S^*)^{p/p'}
        \int_{B(x,\varepsilon)\cap S^*}|f(y)|^p u(y)\,d\rho_\alpha(y), \\
        \|E_{\varepsilon}^0\|_{L^p(u\,d\rho_\alpha)\to L^p(v\,d\rho_\alpha)}
        &\lesssim L_a
        \sup_{x\in S^*}u'(B(x,\varepsilon)\cap S^*)^{1/p'}
        \sup_{y\in S^*}v(B(y,\varepsilon)\cap S^*)^{1/p}.
\end{align*}
Both suprema tend to zero as $\varepsilon\downarrow0$, because the measures
$u'\,d\rho_\alpha$ and $v\,d\rho_\alpha$ are finite and non-atomic on $S^*$.
Thus the diagonal error is small in operator norm.

Second, since $R>s_1$, $a(x)=0$ for $x>R$.  The far-$x$ error is
$$
        E_R^xf(x)
        =-\one_{(R,\infty)}(x)\int_S a(y)K(x,y)f(y)\,d\rho_\alpha(y).
$$
Using H\"older's inequality in $y$ and the size estimate \eqref{eq:CZ-size},
\begin{align*}
        \|E_R^x\|_{L^p(u\,d\rho_\alpha)\to L^p(v\,d\rho_\alpha)}
        &\lesssim \|a\|_\infty u'(S)^{1/p'}
        \left(\int_R^\infty
        \sup_{y\in S}\frac{v(x)}{\rho_\alpha(B(x,|x-y|))^p}
        \drho(x)\right)^{1/p}.
\end{align*}
The last factor tends to zero by Lemma~\ref{lem:tail-integrability}, applied
with $s=p$, $w=v$, and $E=S$.

Third, since $a(y)=0$ for $y>R$, the far-$y$ error is
$$
        E_R^yf(x)=\one_S(x)a(x)\int_R^\infty K(x,y)f(y)\,d\rho_\alpha(y).
$$
For $x\in S$, H\"older's inequality in the $y$ variable gives
$$
        |E_R^yf(x)|
        \le \|a\|_\infty\|f\|_{L^p(u\,d\rho_\alpha)}
        \left(\int_R^\infty |K(x,y)|^{p'}u'(y)\,d\rho_\alpha(y)\right)^{1/p'}.
$$
For $x\in S$ and $y>R$, doubling gives
$\rho_\alpha(B(x,|x-y|))\simeq_\alpha\rho_\alpha(B(y,|x-y|))$.  Therefore
\begin{align*}
        \|E_R^y\|_{L^p(u\,d\rho_\alpha)\to L^p(v\,d\rho_\alpha)}
        &\lesssim \|a\|_\infty v(S)^{1/p}
        \left(\int_R^\infty
        \sup_{x\in S}\frac{u'(y)}{\rho_\alpha(B(y,|x-y|))^{p'}}
        \drho(y)\right)^{1/p'}.
\end{align*}
This tends to zero by Lemma~\ref{lem:tail-integrability}, with $s=p'$,
$w=u'$, and $E=S$.

The approximating kernels do not cut off the endpoint $0$.  Thus there is no
fourth error there, and the only local singularity is the diagonal singularity,
which has already been absorbed by the $\rho_\alpha$-Lipschitz cancellation of
$a$.  Combining the three estimates gives
$$
        \|[a,T]-A_{\varepsilon,R}\|_{L^p(u\,d\rho_\alpha)\to L^p(v\,d\rho_\alpha)}
        \le o_{\varepsilon\downarrow0}(1)+o_{R\to\infty}(1).
$$
Therefore $[a,T]$ is a norm limit of compact operators and is compact.
\end{proof}

\subsection{Sufficiency of VMO}

{ 
The following argument is the half-line version of the compactness
sufficiency argument in \cite[Theorem~1.4]{LL22}.
}

\begin{proposition}\label{prop:compact-suff-conj}
Let $u,v\in A_p(\rho_\alpha)$, let $\eta=(u/v)^{1/p}$, and let
$T$ be as in Theorem~\ref{thm:Bloom-upper-general}.  If
$b\in\VMO_{\eta,\alpha}$, then
$
        [b,T]:L^p(u\,d\rho_\alpha)\to L^p(v\,d\rho_\alpha)
$
is compact.
\end{proposition}

\begin{proof}
By Proposition~\ref{prop:density-VMO}, choose $a_k\in\Crhoc$ such that
$$
        \|b-a_k\|_{\BMO_{\eta,\alpha}}\to0.
$$
The boundedness theorem, Theorem~\ref{thm:Bloom-upper-general}, gives
$$
        \|[b-a_k,T]:L^p(u\,d\rho_\alpha)\to L^p(v\,d\rho_\alpha)\|
        \lesssim \|b-a_k\|_{\BMO_{\eta,\alpha}}\to0.
$$
By Lemma~\ref{lem:smooth-compact}, each $[a_k,T]$ is compact.  Hence
$[b,T]$ is a norm limit of compact operators and is compact.
\end{proof}

\subsection{Necessity of VMO}

{ 
We use a weak-convergence version of the weighted compactness argument in
\cite[Theorem~1.6]{LL22}.  The
point is that the testing functions are carried by the companion intervals,
not by the original intervals where the oscillation is measured.
}

\begin{lemma}\label{lem:weak-null-indicators}
Let $u\in A_p(\rho_\alpha)$.  Let $J_j$ be intervals and let
$F_j\subset J_j$ be measurable sets.  Assume that there is a fixed
$c_0>0$ such that
$$
        \rho_\alpha(F_j)\ge c_0\rho_\alpha(J_j)
        \qquad\text{for all }j.
$$
If, after passing to a subsequence, the intervals satisfy one of the three
conditions
$$
        |J_j|\to0,\qquad |J_j|\to\infty,\qquad \inf J_j\to\infty,
$$
then
$$
        f_j=\frac{\one_{F_j}}{u(F_j)^{1/p}}
$$
converges weakly to zero in $L^p(u\,d\rho_\alpha)$.
\end{lemma}

\begin{proof}
Lemma~\ref{lem:Ap-fixed-portion} gives a constant $c_u>0$, depending only
on $u$ and $c_0$, such that
$$
        u(F_j)\ge c_u u(J_j)
        \qquad\text{for all }j.
$$
Let $h_0\in L^{p'}(u'\,d\rho_\alpha)$, $u'=u^{1-p'}$.  It is enough to prove
$$
        u(F_j)^{-1/p}\int_{F_j}h\,d\rho_\alpha\to0
$$
for bounded $h$ with compact support, since such functions are dense in
$L^{p'}(u'\,d\rho_\alpha)$ and
$$
        u(F_j)^{-1/p}\left|\int_{F_j}(h-h_0)\,d\rho_\alpha\right|
        \le \|h-h_0\|_{L^{p'}(u'\,d\rho_\alpha)}.
$$

Fix a bounded $h$ supported in a compact interval $K\Subset\Rp$. 
If $\inf J_j\to\infty$, then $F_j\cap K=\varnothing$ for all large $j$, and
the integral is eventually zero.

Assume next that $|J_j|\to0$.  
If a subsequence has
$\inf J_j\to\infty$, we are in the first case.  Otherwise, after passing
to a subsequence, the left endpoints of $J_j$ are bounded above.  
Assume $J_j\subset(0,M)$ for a fixed $M<\infty$. Then,
since $\alpha>-1/2$, the local finiteness and non-atomicity of $\rho_\alpha$
give $\rho_\alpha(J_j)\to0$.  Hence $\rho_\alpha(F_j)\to0$.  Absolute continuity
of $|h|^{p'}u'\,d\rho_\alpha$ gives
$$
        u(F_j)^{-1/p}\left|\int_{F_j}h\,d\rho_\alpha\right|
        \le
        \left(\int_{F_j}|h|^{p'}u'\,d\rho_\alpha\right)^{1/p'}\to0.
$$

It remains to treat $|J_j|\to\infty$.  We assume again that the left endpoints of $J_j$ are bounded above.  Since
$|J_j|\to\infty$, the right endpoints tend to infinity.  Choose an interval
$E_0\Subset\Rp$ with $u(E_0)>0$ that is eventually contained in $J_j$ for all large $j$.
The $A_\infty$ estimate applied to $E_0\subset J_j$ gives
$$
        \frac{u(E_0)}{u(J_j)}
        \le C\left(\frac{\rho_\alpha(E_0)}{\rho_\alpha(J_j)}\right)^\delta.
$$
Since $\rho_\alpha(J_j)\to\infty$, it follows that $u(J_j)\to\infty$, hence
$u(F_j)\to\infty$.  Therefore
\begin{align*}
        u(F_j)^{-1/p}\left|\int_{F_j}h\,d\rho_\alpha\right|
        &=u(F_j)^{-1/p}\left|\int_{F_j\cap K}h\,d\rho_\alpha\right|  
        \le
        \left(\int_K |h|^{p'}u'\,d\rho_\alpha\right)^{1/p'}
        \left(\frac{u(F_j\cap K)}{u(F_j)}\right)^{1/p}
        \to0,
\end{align*}
because $u(F_j\cap K)\le u(K)<\infty$.  This proves the weak convergence.
\end{proof}

The next lemma is the compactness testing part of the median construction.  It
is stated separately to make the compactness lower estimate independent of the
boundedness lower-bound proof.

\begin{lemma}\label{lem:compact-median-testing-sets}
Let $T$ be a Calder\'on--Zygmund operator on $(\Rp,|x-y|,\rho_\alpha)$ whose
kernel satisfies the non-degeneracy condition \eqref{eq:kernel-lower}.
Let $I$ be an interval, and let $\widetilde I$ be the corresponding companion
interval from Proposition~\ref{prop:nondegenerate}.
For every real-valued locally integrable $b$ there are measurable sets
$E\subset I$ and $F\subset\widetilde I$ such that
$$
        \rho_\alpha(F)\ge c\rho_\alpha(\widetilde I),
$$
the function $(b(x)-b(y))K(x,y)$ has a fixed sign on $E\times F$, and that 
$$
        \int_E |[b,T]\one_F(x)|\,d\rho_\alpha(x)
        \ge c\int_I |b-b_I|\,d\rho_\alpha.
$$
The constant $c>0$ depends only on $\alpha$ and on the non-degeneracy constants.
\end{lemma}

\begin{proof}
Since $I$ and $\widetilde I$ are separated, no principal value is involved in
$[b,T]\one_F(x)$ for $x\in I$ and $F\subset\widetilde I$.  Let $m$ be a median
of $b$ on $\widetilde I$ with respect to $\rho_\alpha$.  Put
\begin{align*}
        E_+&=\{x\in I:b(x)\ge m\},
        &E_-&=\{x\in I:b(x)<m\}, \\
        F_-&=\{y\in\widetilde I:b(y)\le m\},
        &F_+&=\{y\in\widetilde I:b(y)\ge m\}.
\end{align*}
Then
$$
        \rho_\alpha(F_-),\rho_\alpha(F_+)
        \ge \frac12\rho_\alpha(\widetilde I).
$$
The kernel has one sign on $I\times\widetilde I$.  For $x\in E_+$, the
integrand in $[b,T]\one_{F_-}(x)$ has one sign, and Proposition~\ref{prop:nondegenerate}
gives
\begin{align*}
        |[b,T]\one_{F_-}(x)|
        =\left|\int_{F_-}(b(x)-b(y))K(x,y)\,d\rho_\alpha(y)\right|       
        \ge \frac{c_\alpha}{\rho_\alpha(I)}
             \int_{F_-}(b(x)-b(y))\,d\rho_\alpha(y)                       
        \ge c |b(x)-m|.
\end{align*}
Here we used $\rho_\alpha(F_-)\simeq\rho_\alpha(\widetilde I)\simeq\rho_\alpha(I)$.
The same argument with $E_-$ and $F_+$ gives
$$
        |[b,T]\one_{F_+}(x)|\ge c|b(x)-m|,
        \qquad x\in E_-.
$$
Consequently
$$
        \int_{E_+}|[b,T]\one_{F_-}|\,d\rho_\alpha
        +\int_{E_-}|[b,T]\one_{F_+}|\,d\rho_\alpha
        \ge c\int_I |b-m|\,d\rho_\alpha
        \ge \frac c2  \int_I |b-b_I|\,d\rho_\alpha.
$$
Choosing the larger of the two displayed lower bounds and
renaming the corresponding pair as $(E,F)$ proves the lemma.
\end{proof}

\begin{lemma}\label{lem:compact-testing}
Let $T$ be a Calder\'on--Zygmund operator on $(\Rp,|x-y|,\rho_\alpha)$ whose
kernel satisfies the non-degeneracy condition \eqref{eq:kernel-lower}. Let $u,v\in A_p(\rho_\alpha)$ and
$\eta=(u/v)^{1/p}$.  If $b$ is real-valued and
$$
        [b,T]:L^p(u\,d\rho_\alpha)\to L^p(v\,d\rho_\alpha)
$$
is compact, then $b\in\VMO_{\eta,\alpha}$.
\end{lemma}

\begin{proof}
Compactness implies boundedness, and Theorem~\ref{thm:Bloom-lower-general}
gives $b\in\BMO_{\eta,\alpha}$.  We prove the three vanishing conditions by
contradiction.

Suppose that one of the conditions in Definition~\ref{def:VMO-nu} fails.  Then
there are $\delta_0>0$ and intervals $I_j$ such that
$$
        \frac1{\eta_\alpha(I_j)}
        \int_{I_j}|b-b_{I_j}|\,d\rho_\alpha\ge\delta_0,
$$
and, after passing to a subsequence, one of the three regimes holds:
$$
        |I_j|\to0,
        \qquad |I_j|\to\infty,
        \qquad \inf I_j\to\infty.
$$
Let $\widetilde I_j$ be the companion interval from
Proposition~\ref{prop:nondegenerate}.  Remark~\ref{rem:kernel-nondeg-data} gives
$$
        \rho_\alpha(\widetilde I_j)\simeq\rho_\alpha(I_j),
$$
and preserves the relevant interval regime: $|\widetilde I_j|\to0$ in the
small-interval case, $|\widetilde I_j|\to\infty$ in the large-interval case, and
$\inf\widetilde I_j\to\infty$ in the far-away case.

Apply Lemma~\ref{lem:compact-median-testing-sets} to each $I_j$.  We obtain
sets $E_j\subset I_j$ and $F_j\subset\widetilde I_j$ such that
$\rho_\alpha(F_j)\ge c\rho_\alpha(\widetilde I_j)$, the sign of
$(b(x)-b(y))K(x,y)$ is fixed on $E_j\times F_j$, and
$$
        \int_{E_j}|[b,T]\one_{F_j}(x)|\,d\rho_\alpha(x)
        \ge c\int_{I_j}|b-b_{I_j}|\,d\rho_\alpha
        \ge c\delta_0\eta_\alpha(I_j).
$$
Set
$$
        f_j=\frac{\one_{F_j}}{u(F_j)^{1/p}}.
$$
Then $\|f_j\|_{L^p(u\,d\rho_\alpha)}=1$.  
For the norm lower bound below
we use only the trivial inequality $u(F_j)\le u(\widetilde I_j)$.

Let $H_j$ be the smallest interval containing $I_j\cup\widetilde I_j$.  By the
companion geometry,
$$
        \rho_\alpha(H_j)\simeq\rho_\alpha(I_j)
        \simeq\rho_\alpha(\widetilde I_j).
$$
Then 
Lemma~\ref{lem:Ap-fixed-portion} gives $u(H_j)\lesssim u(I_j)$.  Hence
$$
        u(F_j)\le u(\widetilde I_j)\le u(H_j)\lesssim u(I_j).
$$
Also $E_j\subset I_j$, so $v'(E_j)\le v'(I_j)$.  Lemma~\ref{lem:Bloom-weight-algebra}
therefore gives
$$
        u(F_j)^{1/p}v'(E_j)^{1/p'}
        \lesssim u(I_j)^{1/p}v'(I_j)^{1/p'}
        \lesssim \eta_\alpha(I_j).
$$
Using $\one_{F_j}=u(F_j)^{1/p}f_j$ and H\"older's inequality on $E_j$, we get
\begin{align*}
        c\delta_0\eta_\alpha(I_j)
        &\le \int_{E_j}|[b,T]\one_{F_j}|\,d\rho_\alpha                         
        =u(F_j)^{1/p}\int_{E_j}|[b,T]f_j|\,d\rho_\alpha                         \\
        &\le u(F_j)^{1/p}
        \|[b,T]f_j\|_{L^p(v\,d\rho_\alpha;E_j)}v'(E_j)^{1/p'}                  \\
        &\lesssim \eta_\alpha(I_j)\|[b,T]f_j\|_{L^p(v\,d\rho_\alpha)}.
\end{align*}
Thus
\begin{equation}\label{eq:image-lower}
        \|[b,T]f_j\|_{L^p(v\,d\rho_\alpha)}
        \ge c\delta_0
        \qquad\text{for every }j.
\end{equation}

By Lemma~\ref{lem:weak-null-indicators}, applied with
$J_j=\widetilde I_j$, the sequence $f_j$ converges weakly to zero in
$L^p(u\,d\rho_\alpha)$ after passing to a subsequence if necessary.  A compact
linear operator maps weakly convergent bounded sequences to norm-convergent
sequences, and the norm limit must be the image of the weak limit.  Compactness
of $[b,T]$ would therefore imply
$$
        \|[b,T]f_j\|_{L^p(v\,d\rho_\alpha)}\to0,
$$
contradicting \eqref{eq:image-lower}.  Hence none of the three vanishing
conditions can fail, and $b\in\VMO_{\eta,\alpha}$.
\end{proof}

\subsection{Proof of the compactness theorem}

\begin{proof}[Proof of Theorem~\ref{thm:main-compact}]
As in Subsection~\ref{subsec:proof-bounded}, put
$
        U_\mu=x^{p-2\alpha-1}\mu,
        \
        U_\lambda=x^{p-2\alpha-1}\lambda.
$
Then $U_\mu,U_\lambda\in A_p(\rho_\alpha)$ and
$\nu=(U_\mu/U_\lambda)^{1/p}$.

If $b\in\VMO_{\nu,\alpha}$, Proposition~\ref{prop:compact-suff-conj} applied
to $T=\mathcal R_\alpha$ gives compactness of
$$
        {}[b,\mathcal R_\alpha]:L^p(U_\mu\,d\rho_\alpha)\to
        L^p(U_\lambda\,d\rho_\alpha).
$$
The isometric identity \eqref{eq:norm-comm} transfers this compactness to
$$
        {}[b,R_\alpha]:L^p(\mu\,dx)\to L^p(\lambda\,dx).
$$

Conversely, if $[b,R_\alpha]$ is compact on the Andersen--Kerman Lebesgue setting,
then \eqref{eq:norm-comm} shows that $[b,\mathcal R_\alpha]$ is compact from
$L^p(U_\mu\,d\rho_\alpha)$ to $L^p(U_\lambda\,d\rho_\alpha)$.  The kernel of
$\mathcal R_\alpha$ is non-degenerate by Proposition~\ref{prop:nondegenerate}.
Lemma~\ref{lem:compact-testing} gives $b\in\VMO_{\nu,\alpha}$.
\end{proof}

\bigskip
\section*{Acknowledgments} C. Wen would like to thank Professor Jill Pipher for her discussions at Macquarie University for emphasizing the need to state explicitly the inclusion relations among the weight classes.

J. Li and C. Wen are supported by the Australian Research Council 
DP260100485.  C.-W. Liang is supported by NSTC grant
111-\allowbreak2115-\allowbreak M-\allowbreak002-\allowbreak010-\allowbreak MY5.
L. Wu is supported by the National Natural Science Foundation of China under grant 12201002.

\end{document}